\documentclass[11pt,reqno]{amsart}

\usepackage[papersize={8.5in,11in},width=6in,height=8in,centering]{geometry}
\usepackage{amsfonts}
\usepackage{amssymb}
\usepackage{amsthm}
\usepackage{amsmath}
\usepackage{graphicx}

\usepackage[all]{xy}

\theoremstyle{definition}
 \newtheorem{theorem}{\bf Theorem}[section]
 
 \newtheorem{lemma}[theorem]{Lemma}
 \newtheorem{corollary}[theorem]{Corollary}
 
\theoremstyle{definition}
 \newtheorem{example}[theorem]{Example}
 \newtheorem{remark}[theorem]{Remark}
 
 \newtheorem{proposition}[theorem]{Proposition}
\numberwithin{equation}{section}

\begin{document}

\title[4-dimensional Hamiltonian  manifolds with boundary]{On the existence of  Hamiltonian 4-manifolds  with a contact type boundary}

\author[A.\ Marinković]{Aleksandra Marinković}
\address[]{Matemati\v{c}ki Fakultet, Studentski trg 16, 11
  000 Beograd, Srbija}
\email{aleksandra.marinkovic@math.bg.ac.rs}

%\subjclass{Primary 53D35; Secondary 53D20}

%\keywords{Hamiltonian action, symplectic manifold}

%\date{December 26, 2024}

%%% ----------------------------------------------------------------------

\begin{abstract}
While the Hamiltonian group actions on closed symplectic manifolds have been widely explored throughout the last couple of decades, the study on Hamiltonian group actions on symplectic manifolds with a contact type boundary has started only recently, with the work by  Niederkrüger and the author \cite{MN}. In this note we pursue this study by presenting several methods to construct such Hamiltonian circle manifolds in dimension 4.

\end{abstract}

%%% ----------------------------------------------------------------------
\maketitle
%%% ----------------------------------------------------------------------
%\tableofcontents
\tableofcontents

\section{Introduction}

Over the past several decades Hamiltonian group actions on closed symplectic manifolds have been extensively studied by many authors 
\cite{Atiyah, Audin,AH,Delzant, GS, HK, JR, Karshon,  McDuff, MT, Li, Tolman} 
etc. The profound richness of these studies has captivated researchers, inspiring them to focus on diverse topics:  foundational concepts of Hamiltonian actions, numerous topological properties of  Hamiltonian manifolds, symplectic reduction questions, obstructions for symplectic action to be Hamiltonian, classification results in various cases and many others. The study also extends to closed symplectic orbifolds \cite{KL} and to open symplectic manifolds \cite{KZ}. However,   the foundation   on Hamiltonian group actions on symplectic manifolds with a contact type boundary has been  only recently  established 
in \cite{MN}. While the most of the topics addressed above are yet to be explored, there are some questions completely answered in \cite{MN}. For instance, the contact type boundary of a Hamiltonian circle 4-manifold is always connected  (\cite[Theorem C.]{MN}).

 \medskip

We now briefly emphasise 
some relevant properties of Hamiltonian circle actions that hold  for closed symplectic manifolds as well as for symplectic manifolds with a contact type boundary.

\begin{itemize} 

\item The corresponding Hamiltonian function $H$ is  Morse-Bott and the indices of all critical sets of $H$ are even numbers;

\item critical sets are precisely the fixed point sets of the circle action and they are symplectic submanifolds;

\item the maximum and the minimum are unique;

\item there exists an invariant metric and an invariant almost complex structure $J$ that defines $\mathbb C^*$-action on the manifold. 

\end{itemize}
However, there are still crucial differences.

\begin{itemize} 

\item
A symplectic action on a closed symplectic manifold may not be Hamiltonian, take for instance $T^2$ with a volume form and a rotation. However, 
 every symplectic  action on a symplectic manifold with a contact type boundary is automatically Hamiltonian  (\cite[Theorem A.]{MN});

\item while a Hamiltonian circle action on a closed manifold always has fixed points, the existence of a boundary allows even a free Hamiltonian circle action;

\item in the presence of a contact type boundary,  critical sets of a Hamiltonian function may also have a non empty boundary;

\item in the closed case, the gradient trajectories of $\nabla H$  start and end only in a critical point and invariant $J$-holomorphic curves (i.e. $\mathbb C^*$-orbits) are diffeomorphic to spheres (also called gradient spheres). If there is a contact type  boundary, trajectories can also start or end at the boundary. Therefore, invariant $J$-holomorphic curves are diffeomorphic to spheres, discs or annuli.

%\item   While in the closed case all level sets of a Hamiltonian function are connected, in the boundary case this may not be true. Namely, a small perturbation of the boundary leads to disconnected level sets. However, attaching cylindrical ends to a symplectic manifold with a contact type boundary and extending a Hamiltonian function we are able in \cite{MarinkovicNiederkrueger} to obtain connected level sets in the completion (\cite[Theorem E.]{MarinkovicNiederkrueger}). 

\end{itemize}

\medskip
These interesting comparisons motivate us to keep exploring Hamiltonian $S^1$-manifolds with a contact type boundary.
The next natural step is to explore the existence of these manifolds.
In this note we present  methods to construct such 4-manifolds starting from oriented surfaces (Theorem \ref{theorem1a} and Theorem \ref{theorem2}) as well as constructing new manifolds from already existing once (Theorem \ref{theorem3} and Theorem \ref{theorem4}). We also make the first step toward the classification of Hamiltonian $S^1$-manifolds with a contact type boundary with a free action (Theorem \ref{theorem1b}).
The classification of closed Hamiltonian $S^1$-manifolds in dimension 4 is completely done by Karshon in \cite{Karshon} and it relies on the information encoded in the fixed point sets and the gradient spheres connecting them.
Since in the presence of a contact type boundary, beside gradient spheres, there also exist gradient discs and gradient annuli,  it would be interesting to pursue exploring the classification  of 
Hamiltonian $S^1$-manifolds with a contact type boundary.

\medskip
In Section 2 we introduce the reader with basic properties of these manifolds and in Section 3 we present  basic examples,  the closed unit disc  with a standard symplectic structure 
  (Example \ref{example:disc} a))
and disc cotangent bundle examples, where the Hamiltonian circle action is a lift of an effective circle action on the base (Example \ref{example:cotangent}). Note that, according to Audin \cite[I.3.a.]{Audin}, the only surfaces admitting  such action are $T^2, S^2, \mathbb{RP}^2$ and a Klein bottle. There are also  Hamiltonian 4-manifolds with contact type boundary that are not necessarily Liouville domains. The existence of holomorphic spheres (Example  \ref{example:disc} b)) or closed critical surfaces (Theorem  \ref{theorem2}) guarantees, due to Stoke's,  that the corresponding symplectic form is not  exact.
  
  \medskip  
  
Our first main result in this article is the construction of symplectic 4-manifolds with a convex contact boundary that are equipped with a free Hamiltonian circle action.
The construction is based on defining the set of Legendrian orbits, i.e. the  orbits at the boundary that are tangent to the contact structure.

\begin{theorem}\label{theorem1a}
 For any orientable surface $B$ and any decomposition of $B$ into two homeomorphic surfaces with a common boundary $\Gamma$, there exists a symplectic 4-manifold $S^1\times \tilde B$  with a contact type boundary $S^1\times  B$, where $\partial \tilde B=B$, that admits a free Hamiltonian circle action whose set of Legendrian orbits is precisely $S^1\times \Gamma$.
\end{theorem}

We further  show  that every symplectic 4-manifold with a contact type boundary that admits a free Hamiltonian circle action is necessarily a trivial circle bundle,  what is not true in higher dimensions (Remark \ref{remark:higher}). We also show that every such manifold arises from the construction presented in Theorem \ref{theorem1a}.

 \begin{theorem} \label{theorem1b}
 Every symplectic 4-manifold $(W, \omega)$ with a contact type boundary $V$ and a free Hamiltonian circle action is a trivial circle bundle over a 3-manifold whose boundary is homeomorphic to an orientable surface $B$. Moreover,   the set of Legendrian orbits is a collection of tori $S^1\times\Gamma,$ where $\Gamma$ is the collection of circles that divides $B$ into two homeomorphic surfaces with a common boundary $\Gamma.$
\end{theorem}

 Next, we  study the existence of critical surfaces
 in a Hamiltonian 4-manifold with a convex contact type boundary.
According to Karshon  \cite[Lemma 6.15]{Karshon}, a closed Hamiltonian 4-manifold admits at most two critical surfaces; in this case these two surfaces are  diffeomorphic and the Hamiltonian 4-manifold is a ruled manifold, i.e. it is a sphere bundle over a surface, where the circle action is  a rotation of the fiber. In contrast, a Hamiltonian 4-manifold with a contact type boundary admits at most one critical surface (Corollary \ref{cor:max}). 
The construction of such manifolds is shown in the following theorem.

\begin{theorem}\label{theorem2}
For any  symplectic surface $\Sigma$ (closed or with a non-empty boundary)  with an integral symplectic form $\omega_{\Sigma}$
there exists a symplectic  4-manifold with a contact type boundary that admits a Hamiltonian circle action whose  only critical set is homeomorphic to the surface $\Sigma$
and there are no other critical sets.
\end{theorem}

These manifolds are constructed using circle and disc bundles over surfaces, with the requirement of smoothing the corners when the underlying surface has a non-empty boundary.

\vskip2mm

We further describe some methods of constructing new Hamiltonian 4-manifolds with a contact type boundary starting from already existing ones. 
In Section \ref{section:handle}  we introduce the attachment of a Weinstein 2-handle equipped with a Hamiltonian circle action. The main result is the following.

\begin{theorem}\label{theorem3}
Suppose that a symplectic 4-manifold with a contact type boundary admits a Hamiltonian circle action with  Legendrian orbits. Then, there exists a Weinstein handle attachment such that obtained symplectic 4-manifold with a contact type boundary also admits a Hamiltonian circle action.

\end{theorem}

To prove this theorem, we first equip the standard Weinstein 2-handle with a Hamiltonian circle action. Then, we attach this handle by gluing a Legendrian orbit of the handle (i.e an attaching sphere)  to a Legendrian orbit of the given Hamiltonian manifold, so that the attaching map and the corresponding extension to a collar neighborhood are invariant under the given action. In this way we obtain a global Hamiltonian circle action on the obtained manifold. This action admits one new fixed point and two gradient discs that connect this point with the boundary.  

\vskip2mm

Finally, in Section \ref{section:blow} we briefly recall the procedure of a blow up of an interior critical point, what is also explained in  \cite[Section 6]{Karshon}) for closed Hamiltonian 4-manifolds.
This procedure can be analogously applied to Hamiltonian 4-manifolds with a contact type boundary, by choosing sufficiently small size of the blow up, so that the boundary remains untouched.

\begin{theorem}\label{theorem4}
Suppose that a symplectic 4-manifold with a contact type boundary admits a Hamiltonian circle action with  at least one fixed point (isolated or not). Then, by performing a sufficiently small blow up of a fixed point  one obtains a symplectic 4-manifold with a contact type boundary that also admits a Hamiltonian circle action.

\end{theorem}

%---------------------------------------------------------------------------------------
\section{Properties of Hamiltonian manifolds with a contact type boundary}\label{section:properties}
%-----------------------------------------------------------------------------------------

In this section we review the basic properties of Hamiltonian manifolds with a contact type boundary with an accent on such 4-manifolds.
\vskip2mm
Let $(W,\omega)$ be a symplectic manifold with a convex contact type boundary $V$. Then,  there exists a Liouville vector field $Y$ defined in the neighborhood of the boundary $V$ that is transversal to the boundary and points out of the boundary. The Liouville condition $L_Y\omega=\omega$ and Cartan's formula imply that the vector field $Y$ defines a contact form $\alpha=\iota_Y\omega$ on $V$ where $\omega=d\alpha$.

Suppose $(W,\omega )$ is equipped with a Hamiltonian circle  $S^1=\mathbb R/2\pi \mathbb Z$ action. That is, if $X$ denotes the vector field that generates this circle action, then there is a family of smooth functions $H:W\to\mathbb R$, called Hamiltonian functions, that differ by a constant, such that
$$\iota_X\omega=-dH.$$ 
 As shown in \cite[Theorem B]{MN}, any such Hamiltonian function $H$ is a Morse-Bott function, whose critical sets are precisely the fixed point sets of the action.

In general, a Liouville vector field $Y$ may not be invariant under this action, however, we can always average it in order to obtain an invariant Liouville vector field:
$$Y_{inv}=\int_{S^1}(\theta_{\ast}Y)d\theta.$$

We keep denoting $Y$ this invariant Liouville vector field. Then, the induced contact form $\alpha=\iota_Y\omega$ is also invariant. Thus, 
by  Cartan's formula, it follows that
$0=L_X\alpha=\iota_Xd\alpha+d(\iota_X\alpha)$, i.e.
$$\iota_X\omega=-d(\alpha(X)).$$
Therefore, $\alpha(X)$ is a Hamiltonian function on the boundary $V$ and there is unique Hamiltonian function $H$ on  $W$ whose restriction to the boundary is $\alpha(X)$ (that depends on the choice of an invariant contact form). Through out the article we work with this Hamiltonian function on $W.$

\vskip1mm

Further, by averaging any Riemannian metric on $W$ we obtain an invariant Riemannian metric 

$$g_{inv}=\int_{S^1}(\theta^*g) d\theta.$$
We denote this invariant metric by $g.$
Having such a metric, we may observe the gradient  vector field $\nabla H$ defined by the relation $dH(\cdot)=g(\nabla H,\cdot)$. The corresponding gradient trajectories of $\nabla H$,
denoted by $\Phi_t^{\nabla H}$, can start or end in a critical point of $H$ or at the boundary. 

A Hamiltonian function $H$ is constant on orbits, as $dH(X)=0,$ however, $H$ increases along any trajectory $\Phi_t^{\nabla H},$ as $dH(\nabla H)>0.$

%----------------------------------
\begin{remark}

  If $g$ is an invariant metric,  then, induced almost complex structure $J$ on $TW$ (see\cite[Proposition 2.5.6]{MS})
is also  invariant, i.e.
  $\theta^*J=J,$ for every $\theta\in S^1.$ Note that
$dH(\cdot)=g(\nabla H,\cdot)=\omega(\nabla H,J\cdot)=\omega(J\nabla H, J^2\cdot)=-\omega(J\nabla H,\cdot)$
  thus, from the relation   $dH(\cdot)=-\omega(X,\cdot)$ and the non-degeneracy of the symplectic form, we conclude
  $$X=J\nabla H\hskip2mm\textrm{and}\hskip2mm\nabla H=-JX.$$
  It follow that
  $L_X\nabla H=(L_XJ)X+JL_XX=0$
  and therefore the flows of $X$ and $\nabla H$ commute. The union of any $S^1$-orbit together with all the flow lines of $\nabla H$ along that orbit is a gradient $J$-holomorphic cylinder. The manifold $W$ is foliated by such $J$-holomorphic cylinders. Equivalently, Hamiltonian $S^1$ action and $\nabla H$ defines J-holomorphic $\mathbb C^*$ action on $W$.
   Since any gradient trajectory starts and ends in a critical point or at the boundary, these $J$-holomorphic orbits are diffeomorphic to the spheres, discs or annuli.
Finally,  the stabiliser of all points in one orbit is the same. If the stabiliser is trivial, we call such an orbit a free gradient curve, while if the stabiliser is $\mathbb Z_k$  we call such an orbit a $\mathbb Z_k$ gradient curve. 
\end{remark}
%------------------------------------
\subsection{Invariant collar neighborhood}\label{section:max}
%-------------------------------------

A symplectic manifold $(W,\omega)$ with a contact type boundary $V$ and a Liouville vector field $Y$  admits a boundary collar that is
  diffeomorphic to $(-\epsilon,0]\times V$. This diffeomorphism can be constructed using the flow $(s,p)\mapsto\Phi_s^{Y}(p)$ of the vector field $Y$. Therefore,
  
  \begin{itemize}
  \item $V$ is  identified with $\{0\}\times V$, and $Y$ corresponds to $\partial_s$;
  \item the Liouville form $\iota_Y\omega$ is diffeomorphic to $e^s\alpha$ in this collar  and the
    symplectic structure~$\omega$ is symplectomorphic to
    $d(e^s\alpha)$;
  \end{itemize}
  
  \vskip2mm
  
  A Hamiltonian circle action has a very nice description in this collar neighborhood. Namely,
   if a circle acts symplectically on a manifold $(W,\omega)$ with a contact
type boundary~$V $, then we can choose 
 an invariant Liouville vector field $Y$ and an invariant contact form
$\alpha=\iota_Y\omega$ on $V$ such that the boundary collar of
$W$ is equivariantly diffeomorphic to 

  \begin{equation}  \label{eq:collar}
   ( (-\epsilon, 0] \times V, \omega = d(e^s\alpha)),
  \end{equation}
  
   and such that the circle action agrees with
     \begin{equation}   \label{eq:action}  
\theta\ast (s,p) = (s, \theta \ast p),   \hskip1mm  \textrm{for any }  \hskip1mm \theta\in S^1   \hskip1mm  \textrm{and any}
  \hskip1mm (s,p) \in (-\epsilon,0]\times V. 
  \end{equation}
 The Hamiltonian function~$H$ associated to this circle action and the Liouville
form~$e^s \alpha$ simplifies in this neighborhood to
$$H(s,p) = e^s\, \alpha_p(X_V) .$$

 Choose an invariant Riemannian metric~$g$ on $W$ that agrees with
  \begin{equation*}
    {g}|_{(-\epsilon,0] \times V} = e^s\cdot \bigl(ds^2 \oplus g_V\bigr) \;,
  \end{equation*}
  on the collar $(-\epsilon,0] \times V$, where $g_V$ is any invariant metric on $V.$
  Then, we have the splitting
$TW=TV\oplus Y,$ that is, $g(v,Y)=0,$ for every vector $v\in TV.$     
  
  \vskip2mm

On the collar, the gradient vector filed
  $\nabla H$ is given, with respect to $g$, by
  \begin{equation}\label{eq:nabla_H}
      \nabla H(s,p) =  H|_V(p)\, \partial_s + \nabla H|_V(p). \;
  \end{equation}

\begin{lemma}\label{lemma:sign of H}
For every boundary point $p\in V$ with $H(p) \neq0$ the gradient trajectory $\Phi_t^{\nabla H}(p)$ is non-constant.  If $H(p)<0$ ($H(p)>0$) the trajectory $\Phi_t^{\nabla H}(p)$ starts (ends) in $p.$ Moreover, every trajectory can intersect the boundary only in end points.

\end{lemma}
\begin{proof} 
According to equation (\ref{eq:nabla_H}) it follows that $H(p)>0,$ ($H(p)<0$) for any $p\in V,$ if and only if $\nabla H(p)$
  points transversally out of (into) the  boundary.

Let us show that every gradient trajectory $\Phi_t^{\nabla H}(x)$ that hits  the boundary from the interior of $W$ in a positive time direction hits it transversally in a point and it stops. The proof is analogous for the negative time direction.
If the trajectory $\Phi_t^{\nabla H}(x),$ for some $x$ in the interior of $W,$ hits the boundary in time $t_0>0$ in a point $p\in V$ tangentially (i.e. $\Phi_{t_0}^{\nabla H}(x)=p$), then $g_p(Y,\nabla H)=0$. Thus,   $dH_p(Y)=0$
and, using the relations 
$dH(Y)=\omega(Y,X)=\alpha(X)=H$ on  $V$
it follows that
 $H(p)=0.$ However, then, $H\circ\Phi_t^{\nabla H}(x)<0$
for all time $t<t_0$, as the value of $H$ increases along any trajectory. Therefore, $\nabla H(\Phi_t^{\nabla H}(x))$ points into $W$, for $t<t_0,$ what is not possible, as we assumed that the trajectory $\Phi_t^{\nabla H}(x)$  hits  the boundary in a positive time direction in $p$, i.e. 
 $\nabla H(\Phi_t^{\nabla H}(x))$
  points  out of the boundary, for $t<t_0$ close to $t_0$.  
  Similarly, for the negative time direction (i.e. when $t>t_0$).
  
  \end{proof}

    We also have the following result.

\begin{lemma}\label{lemma:H=0}
For every boundary point $p\in V$ with $H(p)=0$  the gradient trajectory $\Phi_t^{\nabla H}(p)$ is constant.
\end{lemma}
\begin{proof} Let us show that there are  two types of boundary points where $H$ vanishes, fixed points of the circle action and the regular points whose orbits are isotropic (i.e. tangent to the contact structure). 

If $p$ is a fixed point of the circle action, then $X(p)=0$ and then $\iota_X\omega=-dH=0$ in $p.$ Further, from the relation $dH(Y)=\omega(Y,X)=\alpha(X)=H$ it follows that $H(p)=0.$
It is clear that the flow $\Phi_t^{\nabla H}(p)$ is constant.

If $p$ is a regular point (i.e. $X(p)\neq0$) with $H(p)=0$, then $\alpha(X_p)=0$. That is, $X(p)$ is tangent to the contact structure $\ker\alpha.$ Similarly, for any other point $\theta\ast p$ in the same orbit it holds $\alpha(X_{\theta\ast p})=\theta^{\ast}\alpha(X_p)=0$ and therefore the whole orbit in $p$ is tangent to $\ker\alpha.$
As shown above, the trajectory $\Phi_t^{\nabla H}(p)$ cannot enter the interior of $W$ nor it can move along the boundary, therefore it has to stay constant.
\end{proof}

\begin{corollary}
Gradient trajectories of $\nabla H$ never converge to a boundary critical point. 

\end{corollary}
The behaviour of the gradient flow near the boundary is essential in understanding local and global extrema of the Hamiltonian function. 
Critical sets can be closed, located in the interior of $W,$  or with a non empty boundary that is contained in the boundary of $W$ (see \cite[Theorem B]{MN}). In particular, there does not exist  boundary critical set that is not the boundary of some interior critical set.

\begin{proposition}\label{prop:max}
The maximum and minimum of $H$ are unique and at least one of them is attained at the boundary. 
\end{proposition}
\begin{proof}
According to the properties of the flow described above we can decompose $W$ in the following way 

$$W=\sqcup_{j=1}^mW^s(C_j)\sqcup W^s(\partial W)\sqcup \partial^0W,$$
where

$$
  W^s(C_j) := \bigl\{p\in W\bigm|\; \lim_{t\to \infty} \Phi_{\;
    t}^{\nabla H}(p) \in C_j \bigr\},$$
    is the stable manifold of the critical set $C_j,$ $j=1, \ldots, m,$    
 $W^s(\partial W)$
 is the set of trajectories that end at the boundary (in regular points) and 
 $\partial^0W$
 is the set of isotropic orbits at the boundary.

 Since every critical set $C_j$ is even dimensional manifold, the corresponding index $i^-(C_j)$, i.e. the dimension of the negative eigenspace of the corresponding Hessian, is also even and
 $\dim W^s(C_j)=\dim C_j+i^-(C_j)$, it follows that all stable manifolds are even dimensional. In particular, if the maximum of $H$ is attained on the critical set $C_{max}$ then
$\dim W^s(C_{max})=\dim W$ and other stable manifolds are codimension $\geq2$.
Next, $\partial^0W$ is either empty or it is a hypersurface at the boundary corresponding to the zero level set of $H$.

Finally, let us show that $W^s(\partial W)$ is either empty or an open connected submanifold. 
If $W^s(\partial W)$ is empty, then, according to the decomposition of $W$ into even dimensional submanifolds, it follows that there is unique full dimensional stable manifold and it is the stable manifold of the maximum.
Suppose $p_0\in W^s(\partial W)$.
If $p_0\in\partial W$ then $H(p_0)>0$ and there is a neighborhood $U\subset\partial W$
of $p$ such that $H|_U>0.$ 
Then, we can choose a small  $\epsilon > 0$ such that $(t,p)\mapsto \Phi^{\nabla H}_t(p)$
  defines a diffeomorphism of $(-\epsilon, 0]\times U$ onto a
  neighborhood of $p_0$ in $W$.  Thus, there exists an open
  neighborhood of $p_0$ that lies in $W^s(\partial W)$.
  If $p_0\in W^s(\partial W)$ is an interior point of $W$, then
  let $\gamma$ be the gradient trajectory starting at $p_0$
  intersecting $\partial W$ at time $t = T$. As above, $\gamma(T)$ has a small
  open neighborhood~$U$ that lies in $W^s(\partial W)$.  If
  $T'\in (0,T)$ is chosen sufficiently close to $T$, then
  $p_1 = \gamma(T')$ will lie in $U$.  Then, there are
  open neighborhoods~$U_0$ of $p_0$ and $U_1$ of $p_1$ such that
  $\Phi^{\nabla H}_{T'}$ is a diffeomorphism between $U_0$ and $U_1$.
  After possibly shrinking the size of $U_1$, we may assume that $U_1$
  lies in $U$.  Replace $U_0$ by $U_0 := \Phi^{\nabla H}_{-T'}(U_1)$,
  then by construction $U_0$ is an open neighborhood of $p_0$ that lies in
  $W^s(\partial W)$.  
   Therefore, if  $W^s(\partial W)$ is non-empty, then it is an open
  subset. In this case, due to the decomposition of $W$, it follows that $W^s(\partial W)$ is connected and there is no critical set where the maximum is attained (i.e. stable manifolds are of codimension at least 2).
  
  If we analogously observe the unstable manifolds, we conclude that either minimum or maximum has to be attained at the boundary, since the sets $W^s(\partial W)$ and 
  $W^u(\partial W)$    cannot be both empty.
  
\end{proof}  

\begin{corollary}\label{cor:max}
If a Hamiltonian 4-manifold with a contact type boundary admits a critical surface, then it is a maximum or a minimum of the Hamiltonian function. Moreover, there can be at most one critical surface.
\end{corollary}
\begin{proof}
According to  the local model formula in \cite[Corollary A.7 ]{Karshon}, it follows that
if there is a critical surface in a Hamiltonian 4-manifold then it must be a local maximum or minimum. 
Therefore, according to Proposition \ref{prop:max} there is at most one critical surface.

\end{proof}

\begin{remark} \label{remark:boundary} According to the decomposition of a Hamiltonian manifold with a convex contact boundary,
it follows that  
there can be  at most two boundary components of $\partial W$, one where gradient trajectories start and the other where gradient trajectories end. 
In particular, if there exists critical set where the minimum or the maximum of the Hamiltonian function is attained, then the boundary is connected.
We address that, according to
\cite[Theorem C]{MN}, in dimension 4 the boundary is always connected. This relevant property will be used in the proof of Theorem \ref{theorem1b}.
\end{remark}

%-------------------------------------------
\subsection{Legendrian orbits on the boundary of Hamiltonian 4-manifolds}
%----------------------------------
Let $(W,\omega)$ be a symplectic 4-manifold with a contact type boundary.
The orbits at the boundary that are tangent to the contact structure are called Legendrian orbits. 
The importance of Legendrian orbits was first established by Lutz in \cite{Lutz} who classified many non contactomorphic contact structures on some 3-manifolds by constructing free circle actions with  Legendrian orbits.

\begin{lemma}\label{lemma:legendrian_tori1}
A Hamiltonian manifold $(W^4,\omega)$ with a contact type boundary admits
Legendrian orbits if and only if both extrema are attained at the boundary $V$.

\end{lemma}
\begin{proof}
If both extrema are attained at the boundary, then the set of trajectories that start at the boundary and the set of trajectories that end at the boundary are both non-empty. According to the relation of the direction of $\nabla H$ and the sign of $H$ (see Lemma \ref{lemma:sign of H}), at the starting points we have $H<0$ and at the ending points we have $H>0.$ As the boundary is connected (\cite[Theorem C]{MN}), there must exist boundary points $\Sigma$ where $H$ vanishes. Moreover, since both extrema are attained at the boundary, it follows that there are no critical surfaces and, in particular, there are no boundary critical points. Therefore, all points of $\Sigma $ are regular points of $H.$ Thus, $X$ does not vanish along $\Sigma $. Finally, since 
for every point $p\in V$ whose orbit is Legendrian it follows that $H(p)=\alpha(X_p)=0$.
Therefore,  all orbits of $\Sigma$ are isotropic, and thus, Legendrian.

On the other hand, if there is a set $\Sigma$ of Legendrian orbits, then $H|_V$ changes the sign in the neighborhood of $\Sigma.$ Namely, if $H>0$ on the boundary $V$, then, since $H=0$ on $\Sigma$, it follows that $\Sigma$ is a local minimum of $H$ and thus $dH$ would vanish along $\Sigma$. Let us show that
$dH$ does not vanish on $\Sigma$. This can be seen  using Cartan's formula
$0=L_X\alpha=\iota_Xd\alpha+dH.$
If $dH$ vanishes at some point $p\in\Sigma$, then $\iota_Xd\alpha(p)=0.$ However, since $\alpha_p(X)=0$ it follows that  $X(p)=0,$ i.e. $p$ is a fixed point, what is not possible.
Similarly, it cannot happen that $H<0$ everywhere on $V$ away from $\Sigma.$ Thus, $H$ changes the sign along the boundary. It follows that both sets of the trajectories that start at the boundary and of the trajectories that end at the boundary are non-empty. Threfore, both maximum and minimum are attained at the boundary.
\end{proof}

Note that, as explained above, the existence of Legendrian orbits exclude the existence of boundary fixed points. Alternatively, this can be concluded from     \cite[Lemma IV 19.]{Nie}, 
that states that existence of both Legendrian orbits and fixed points implies that the contact structure is overtwisted. In our case, the contact structure on the boundary of $(  W,  \omega)$ is strongly fillable, due to the existence of a Liouville vector field near the boundary pointing out of the boundary.

\begin{lemma}\label{lemma:legendrian_tori}
For any Hamiltonian manifold $(W^4,\omega)$ with a contact type boundary 
the set of Legendrian orbits is either empty or a collection of tori.

\end{lemma}

\begin{proof}
If $H:W\to\mathbb R$ denotes a Hamiltonian function, then the set of Legendrian orbits $\Sigma$, if non-empty,  is precisely the set  $H^{-1}(0)\cap\partial W.$
In order to see that the set $\Sigma$ is a closed submanifold, it is enough to show that zero is a regular value of $H.$ This is true, since, as explained above, $dH$ does not vanish on $\Sigma$.

It is left to show that $\Sigma$ is a collection of tori. The orbit space of a small neighborhood $\mathcal N(\Sigma)$ of $\Sigma$ along the boundary $V$ is a smooth surface $\mathcal N(\Sigma)/S^1$. Since a Hamiltonian function $H$ is constant on orbits, it induces a smooth function $H_{red}$ on this surface. The zero set of $H_{red}$ is 1-dimensional submanifold, thus, a family of circles. The set of Legendrian orbits $\Sigma$ is a   principal circle bundle over this family of circles. Therefore, $\Sigma$ is a family of tori.
\end{proof}

We remark that the statement of previous lemma holds for any contact 3-manifold with a Hamiltonian circle action, as explained in     \cite[Lemma IV 6.]{Nie}.

\section{Basic examples}\label{section:examples}

Here are several  basic examples of Hamiltonian 4-manifolds with a contact type boundary.
In the next sections we will see that some of them can be constructed by using different methods.

\begin{example}\label{example:disc}
a) Consider $(D^4, \omega_{st}=\frac{i}{2}\sum_{k=1}^2dz_k\wedge d\overline{z}_k)$ with a Liouville vector field  $Y=\frac{1}{2}\sum_{k=1}^2(z_k\frac{\partial}{\partial z_k}+\overline{z}_k\frac{\partial}{\overline{z}_k}))$ and with the Hamiltonian circle action defined by the (isotropy) weights $(m,n):$
 $$e^{i\theta}*(z_1,z_2)\mapsto(e^{im\theta}z_1,e^{in\theta}z_2),$$
where $m,n\in\mathbb{Z}$ are relatively prime (in order to have an effective action). The generator if this action is the vector field 
$X=im(z_1\frac{\partial}{\partial z_1}-\overline{z}_1\frac{\partial}{\overline{z}_1}))+in(z_2\frac{\partial}{\partial z_2}-\overline{z}_2\frac{\partial}{\overline{z}_2}))$. Therefore, the corresponding moment map is given by 
$$H(z_1,z_2)=\frac{m}{2}|z_1|^2+\frac{n}{2}|z_2|^2.$$

 If $mn\neq0$ then there exists one isolated critical point $(0,0).$ 
 If $m,n>0$ (or $m,n<0$) then $H$ attains its global minimum (maximum) on $(0,0)$ whose index is equal to $0$ (4), while if $mn<0$ then the index of $(0,0)$ is equal to 2. The standard metric $g$, defined by $g(\partial z_k,\partial \bar z_k)=1/2,$ 
 $k=1,2,$ and all other inner products are zero, is invariant under the given action.
 The set of points $z_2=0$ form a  $\mathbb Z_m$-gradient disc and  the set of points $z_1=0$ form  a $\mathbb Z_n$ gradient disc.

 If $mn=0$ then the set of critical points is a 2-disc and the non-zero weight has to be equal to $1$ or $-1$ in order to have an effective action.
Therefore, every point in this disc  is connected to the boundary with a free gradient disc. If $n=0$ and $m=1$ ($m=-1$),  the minimum (maximum) of $H$ is attained at the critical surface $z_1=0$ and the maximum (minimum) of $H$ is attained at the boundary orbit $|z_1|=1.$

b) Suppose $mn\neq0$ and  perform a symplectic blow up of $D^4$ of a size $\epsilon,$ detailed explained in Section \ref{section:blow}. Then, there is an induces Hamiltonian circle action on $D^4\sharp\mathbb{CP}^2$ with the standard symplectic form.
If $mn=1$ then every point of the exceptional sphere is fixed  under the circle action. Thus, the set of fixed points in $D^4\sharp\mathbb{CP}^2$ is a 2-sphere. 
On the other hand, if $mn\neq1$ then there are two fixed points in $D^4\sharp\mathbb{CP}^2$ and they are connected with the exceptional sphere that is also a $\mathbb Z_{|m-n|}$ gradient sphere. 
In particular, the symplectic form on $D^4\sharp\mathbb{CP}^2$ is not exact, due to the existence
of a closed symplectic surface.

\end{example}

\begin{example}\label{example:second}

 Consider $(S^1\times D^3, dh\wedge dt+\frac{i}{2}dz\wedge d\overline{z})$ a symplectic manifold wth a contact type boundary $S^1\times S^2$ where the Liouville vector field is given by $Y=h\frac{\partial}{\partial h}+\frac{1}{2}(z\frac{\partial}{\partial z}+\overline{z}\frac{\partial}{\overline{z}})).$
Hence the contact boundary is $(S^1\times S^2=S^1\times\{|z|^2+h^2=1\}, \ker(hdt+\frac{i}{4}(zd\overline{z}-\overline{z}dz)))$. A circle action

$$e^{i\theta}\ast(e^{it},z,h)\mapsto(e^{i(t+k\theta)},e^{im\theta}z,h),$$ 
where the integers $k$ and $m$ are relatively prime, in order to get an effective action,
is Hamiltonian. 
The corresponding Hamiltonian function is given by
$$H(e^{it},z,h)=kh+\frac{m}{2}|z|^2.$$
The metric $g$ defined by
$g(\partial t, \partial t)=g(\partial h, \partial h)=1,$ $g(\partial z,\partial \bar z)=1/2$ and all other inner product are zero, is invariant under this action. 
The annulus $S^1\times \{h\in[-1,1]\}$ is a $\mathbb{Z}_k$-annulus with respect to the given metric.

In particular, if $k=1,m=0$, the action is free, while if $k=0, m=1$ there exists one critical surface $S^1 \times [-1,1].$ These particular examples will be constructed in Section \ref{section:free} and Section \ref{section:surface}.
\end{example}

\begin{example}\label{example:cotangent} \emph{Cotangent bundle examples.}
For any closed Riemannian manifold $(L,g)$ the corresponding cotangent bundle $T^*L$ can be equiped with a symplectic form $d \lambda_{can},$
where $\lambda_{can}$ is a canonical Liouville 1-form. Moreover, the cosphere bundle over $L$ (where each fiber is a unit sphere with respect to $h$) is a hypersurface of contact type in $T*L$ and, therefore, the disc contangent bundle can be equipped with a symplectic structure whose boundary is convex contact.

Every diffeomorphism on a closed Riemannian manifold $(L,g)$ lifts to a symplectomorphism on the cotangent bundle $(T^*L, d\lambda_{can})$.
Therefore, a smooth circle action  on $L$ induces symplectic circle action on $T^*L.$ 
Moreover, this action is a
Hamiltonian action on $T^*L$ and we now present the corresponding Hamiltonian function in local coordinates. 
Denote by $q_j$ the coordinates on $L$ and by  $p_j$ the coordinates on each fiber $T^*_qL$. Then, the canonical Liouville form is given by
$\lambda_{can}=\sum_j p_jdq_j.$ Denote by  $q_j(t)$, $t\in S^1$, the orbits on the circle action on $L$. 
The Hamiltonian function on $T^*L$ is given by
  \begin{equation*}
   H(q_j,p_j)=\sum_j p_jq'_j(t).  \end{equation*}
 In particular, $H$ vanishes along fixed point sets.
 
Moreover, if  $g$ is an invariant metric on $L$ then, there is a well defined Hamiltonian action on the disc cotangent bundle, denoted by $D^*L$.

According to Audin \cite[I.3.a]{Audin} the only surfaces admitting 
 circle actions by diffeomorphisms are a torus, a sphere, a real projective space and a Klein bottle. 
We now explore these circle actions and the corresponding Hamiltonian actions on the disc cotangent bundles.

\emph{The torus.}  The circle acts on $T^2$  by rotations 
$$e^{i\theta}*(e^{it_1},e^{it_2})\rightarrow (e^{i(m\theta+t_1)}, e^{i(n\theta+t_2)}),$$
 where $m$ and $n$ are relatively prime, in order to have an effective action. This action is free and it lifts to a free Hamiltonian action on 
the disc cotangent bundle $(D^*T^2=T^2\times D^2_{(x,y)},\lambda_{can}=xdt_1+ydt_2).$ The corresponding Hamiltonian function is given by
$H(e^{it_1},e^{it_2},x,y)=mx+ny.$ For more on free Hamiltonian circle actions see Section \ref{section:free}.

\emph{The sphere.}
The circle acts on $S^2$ by rotations around fixed axis.
There are two fixed points, the intersection of $S^2$ and the axises, and all other orbits are free. This action lifts to a Hamiltonian action on 
$D^*S^2$, with two isolated fixed points. 
As explained above, a moment map vanishes along fixed points. In particular, there are no gradient trajectories connecting them, since $H$ increases along these trajectories. The index of both critical points is equal to 2.

\emph{The real projective space.}
The real projective space $\mathbb{RP}^2$ is a quotient space of the 2-sphere under the antipodal action of $\mathbb Z_2$.
Hence, the circle action on $\mathbb{RP}^2$ is induced by the circle action on $S^2.$ 
Two fixed points in $S^2$ are antipodal, so they represent  one fixed point in $\mathbb{RP}^2$ and every point on the equator of $\SS^2$ identified with its antipodal point has $\mathbb{Z}_2$ stabilizer. Hence, there is a circle orbit in $\mathbb{RP}^2$ with $\mathbb Z_2$ stabilizer. As above this action lifts to the Hamiltonian circle action on the  cotangent bundle $T^*\mathbb{RP}^2.$
There is an isolated fixed point in the zero section and the orbit with $\mathbb{Z}_2$ stabilizer in $\mathbb{RP}^2$ lifts to the orbit in the zero section with $\mathbb Z_2$ stabilizer. Since the gradient flow preserves the stabilizer, we conclude that there is one $\mathbb{Z}_2$ annulus.
Existence of a gradient annulus implies that the index of the isolated fixed point is equal to 2.

\emph{The Klein bottle.}
The Klein bottle $K$ is defined as the  quotient space of $\mathbb R^2$ under the transformations $(x,y)\mapsto(x+1,y)$ and $(x,y)\mapsto (-x,y+1).$
There is a well defined $S^1$-action on the Klein bottle given by $\theta*(t_1,t_2)\mapsto(t_1,\theta+t_2)$. The orbits $t_1=0$ and $t_1=\frac{1}{2}$ have $\mathbb{Z}_2$ stabilizer. All other points are free. This action lifts to the Hamiltonian circle action in the cotangent bundle. 
The orbits $t_1=0$ and $t_1=\frac{1}{2}$ lift to the orbits with $\mathbb{Z}_2$ stabilizer in the zero section of the cotangent bundle.  Since the gradient flow preserves the stabilizer  we obtain two $\mathbb{Z}_2$ annuli.

\end{example}

%------------------------------------------------------------------------
\section{Existence of a free Hamiltonian circle action}\label{section:free}
%------------------------------------------------------------------------

\emph{Proof of  Theorem \ref{theorem1a}}.
 Let $B$ be any orientable surface. Take a decomposition of $B$ into two homeomorphic orientable surfaces 
$B_+$ and $B_-$  with the same boundary $\Gamma.$ Then $\Gamma$ is a collection of circles.
Consider the manifold $\tilde{W}=S^1\times [-1,1]\times B_+$ with the $S^1$-action given by the rotation on $S^1$-coordinate. 
\vskip1mm
The annulus $(S^1_t\times[-1,1],dt\wedge ds)$ is a Liouville domain, where the  Liouville vector field is given by $s\partial s.$ Note that  this Liouville vector field is equal to $\nabla f_1$, where the function $f_1:T^*S^1\rightarrow \mathbb{R}$ is given by $f_1(t,s)=s^2/2$ and the metric is defined by $dt\wedge ds(\cdot, J\cdot)$, where $J$ is an almost complex structure
 $J\partial t=\partial s$. Note that 
$S^1\times[-1,1]=f_1^{-1}[0,1/2]$ and 
$dt\wedge ds=-d(df_1 \circ J),$ i.e. $S^1\times[-1,1]$ is a Stein domain.

\vskip1mm
A surface $B_+$ with the boundary $\Gamma$ also admits a structure of a Liouville domain that we now describe.
 Use an embedding of $B_+$ in $\mathbb{R}^3$ to define a height function 
$h:B_+\rightarrow\mathbb{R}$. We then slightly perturb $h$ in order to get a Morse function, that we still denote by $h$. 
Since the surface $B_+$ is classified by the genus $g$ and the number of boundary components $k$, a function $h$ can be chosen in such a way that it admits one critical point of index zero and $2g+k-1$ critical points of index 1.  Moreover, it can be chosen in such a way that it is constant on the boundary of $B_+.$
Define a function
$f_2:B_+\rightarrow \mathbb{R}$ by $f_2=e^{ch},$
where $c>0$  will be defined later.  Since
$df_2=ce^{ch}dh$ then
$$d(-df_2\circ i)=-c^2e^{ch}dh\wedge dh\circ i-ce^{ch}d(dh\circ i)),$$
where $i$ stands for the standard almost complex structure on $\mathbb C \cong \mathbb R^2.$
The first summand is positive far from the critical points of $h$. This can be seen from the relation
$dh\wedge dh\circ i(v,iv)=dh(v)dh(i^2v)<0.$
The second summand is positive close to the critical points, as the Hessian of $h$ is non-degenerate in the neighborhood of critical points, i.e. $-d(dh\circ i)>0.$ Far from critical points the second summand is bounded. Therefore, for $c>0$ large enough it holds that $d(-df_2\circ i)$ is a symplectic form on $B_+$ and $\nabla f_2$ is the corresponding Liouville vector field pointing out of the boundary of $B_+$. Note that  the boundary of $B_+$ is a level set of $f_2.$
Finally,  rescale the function $f_2$  (and denote it again by the same letters) so that $f_2:B_+\rightarrow[0,1/2]$ and $\partial B_+=f_2^{-1}\{1/2\}$.

\vskip1mm
 Next,  extend functions $f_1$ and $f_2$ trivially  to the functions $\tilde f_1, \tilde f_2: \tilde{W}\to \mathbb R$. Then, $\omega=\pi_1^*dt\wedge ds+\pi_2^*d(-df_2\circ i)$  is a symplectic form on $\tilde{W}$, where $\pi_1:\tilde{W}\to S^1\times [-1,1]$ and 
  $\pi_2:\tilde{W}\to B_+$ are projections, and
 $\nabla(\tilde f_1+\tilde f_2)$ is a Liouville vector field with respect to $\omega.$ Then, 
$$(W=(\tilde f_1+\tilde f_2)^{-1}([0,1/2]), \omega|_W)$$
 is a symplectic manifold with the boundary $V=(\tilde f_1+\tilde f_2)^{-1}(1/2).$ 
 The symplectic form $\omega|_W$ is exact since the Liouville vector field $\nabla(\tilde f+\tilde g)$ is globally defined.

 Let us show that the boundary $V$ is of contact type, i.e. that $\nabla(\tilde f_1+\tilde f_2)$ points transversally out of $W$ along $V$. Note that $TV=\textrm{ker} d(\tilde f_1+\tilde f_2)|_V,$ so we have to show that 
 $d(\tilde f_1+\tilde f_2)(\nabla(\tilde f_1+\tilde f_2))>0$ along  $V.$ Using linearity of a differential and a fact that
 $d\tilde f_1(\nabla\tilde f_2)=d\tilde f_2(\nabla \tilde f_1)=0$ it is left to show that $d\tilde f_1(\nabla \tilde f_1)+d\tilde f_2(\nabla\tilde f_2)>0.$ From above construction we have 
 $d\tilde f_1(\nabla \tilde f_1),d\tilde f_2(\nabla\tilde f_2)\geq0.$ 
  If $d\tilde f_1(\nabla\tilde f_1)=0$ at some point $p\in V$ that means that $s$-coordinate of $p$ vanishes and therefore $\tilde f_1(p)=0.$ Since $(\tilde f_1+\tilde f_2)(p)=1/2$, then $\tilde f_2(p)=1/2$. Thus, $d\tilde f_2(\nabla\tilde f_2)(p)>0$, as $\nabla f_2$ is transversal to the boundary $f_2^{-1}(1/2)$ pointing out of the boundary.
 We conclude that $V$ is a contact type boundary. 
  
  Moreover, the $S^1$-action on $\tilde{W}$ given as the rotation on $S^1$-coordinate is a Hamiltonian action with a Hamiltonian function $H(t,s,x)=s,$ for any $(t,s,x)\in \tilde W$.
Since both functions $\tilde f_1$ and $\tilde f_2$ are invariant under this  action, there exists a well defined  restriction of this action to a Hamiltonian action on $W$. 

It is left to show that $V$ is diffeomorphic to $S^1\times B.$ Note that to every point in $V$ we assign coordinates $(t,s,x)\in V\cap (S^1\times [-1,1]\times B_+)$. Thus, a desired diffeomorphism can be given by $(t,s,x)\mapsto\begin{cases}
(t,x), & s\geq 0\\
(t,\varphi(x)), & s<0
\end{cases},$ where $\varphi:B_+\to B_-$ is a given diffeomorphism that is identity on the boundary.

Finally, the set of Legendrian orbits in $V$ is the set where $H$ vanishes, that is, the set of points where $s=0.$ For these points we have $\tilde f_1=0$ and therefore $\tilde f_2=1/2.$ However, 
$\Gamma=\partial B_+=\tilde f_2^{-1}\{1/2\}$. Thus, the set of Legendrian orbits is precisely $S^1\times\Gamma.$

The proof of Theorem  \ref{theorem1a} is completed.

\begin{remark}\label{decompose}
 Let us now see in how many ways we can decompose an orientable surface $B$ of genus $g$ into two diffeomorphic parts
 $B_+$ and $B_-$ with the common boundary $\Gamma$. The explanation is due to Pierre Dehornoy. 
 The surfaces $B_+$ and $B_-$ are uniquely given by the genus $h,$ where $2h\leq g$ and by the number $k$ of boundary components. The Euler characteristic of $B$ is equal to $2-2g,$ while the Euler characteristic of $B_{\pm}$ is equal to $2-2h-k.$ The later can be computed using the appropriate simplicial decomposition of a surface with genus $h$ and then removing $k$ two cells. Now, gluing back $B_+$ and $B_-$ along $\Gamma$, we obtain a closed surface with Euler characteristic $2(2-2h-k).$ Indeed, we can chose a simplicial decomposition of the glued surface obtained as the union of cells of the simplicial decompositions of $B_+$ and $B_-$ just described. Since the surface obtained is $B$
we conclude $2-2g=4-4h-2k.$ That is
$$2h+k=1+g.$$
Hence, the decomposition of a surface $B$ with a genus $g$ can be given by any choice of numbers $h\geq0$ and $k\geq1$ that satisfy the above equality.
\end{remark}

\begin{example}
If $B$ is a sphere or a torus, then, according to Remark \ref{decompose}, there is unique  decomposition of $B$ into two diffeomorphic parts.

a) If $B=S^2$ then, following the above construction, one obtains $V=S^1\times S^2$ and $W=S^1\times D^3$ with a symplectic form
$dh\wedge dt+\frac{i}{2}dz\wedge d\overline{z}$, a Liouville vector field  $Y=h\frac{\partial}{\partial h}+\frac{1}{2}(z\frac{\partial}{\partial z}+\overline{z}\frac{\partial}{\overline{z}})$ and a Hamiltonian circle action
 $$e^{i\theta}\ast(e^{it},z,h)\mapsto(e^{i(t+\theta)},z,h).$$
Since this action is generated by the vector field
$X=\frac{\partial}{\partial t}$, the corresponding Hamiltonian function is given by
$$H(e^{it},z,h)=h.$$
The set of  Legendrian orbits is the torus $S^1\times\{h=0\}\subset S^1\times S^2$. Note that this example is a special case of Example \ref{example:second}.

b) If $B=T^2$ one obtains  $V=S^1\times T^2$ and $W=S^1\times S^1\times D^2$ with a symplectic form 
 $dx\wedge dt_1+dy\wedge dt_2$,  a Liouville vector field $Y=x\frac{\partial}{\partial x}+y\frac{\partial}{\partial y}$ and
a Hamiltonian circle action
$$e^{i\theta}\ast(e^{it_1},e^{it_2}, z)\mapsto(e^{i(t_1+\theta)},e^{it_2},z).$$
The corresponding Hamiltonian function is  given by
 $$H(t_1,t_2,z=x+iy)=x.$$
  The set of Legendrian orbirs is the union of two tori $T^2\times\{x=0,y=1\}$ and  $T^2\times\{x=0,y=-1\}$.
\end{example}
\vskip2mm
We now prove Theorem \ref{theorem1b} that states that any symplectic 4-manifold with a contact type boundary that admits a free Hamiltonian circle action is a trivial bundle. From this theorem also follows that any such manifold can be obtained from the construction presented in Theorem \ref{theorem1a}.

\vskip2mm
\emph{Proof of  Theorem \ref{theorem1b}}.
 In order to show that $W$ is a trivial bundle it is enough to find a global section $\tilde{\sigma}:W/S^1\rightarrow W$. Since the action is free, then the quotient $W/S^1$ is a smooth manifold with boundary $V/S^1.$ The action is also free on the boundary, thus, the quotient $V/S^1$ is homeomorphic to a smooth surface $B$ and it  induces the orientation from $V.$
\vskip1mm
Denote by $\alpha$ an invariant contact form on the boundary and by $H=\alpha(X)$ the corresponding Hamiltonian function that extends to $W.$ Chose an invariant Riemannian metric on $W$ as in Section 
\ref{section:max} and define the gradient flow of $\nabla H$ with respect to this metric.
Since the action is free,  there are no critical points of the function $H$ and thus, every gradient trajectory has to start and end at the boundary. 
As explained in Lemma \ref{lemma:sign of H}, in the starting points we have $H<0$ and in the ending points $H>0.$ Since the boundary is connected
(\cite[Theorem C]{MN}), the set of boundary points where $H$ vanishes  is nonemty. This is presicely the set of Legendrian orbits. Therefore, the gradient flow $\Phi_t^{\nabla H}$ defines a homeomorphism $\tilde\varphi$ between the sets
 $V_-=\{H|_V\leq0\}$ and $V_+\{H|_V\geq0\}$, that is equal to identity on the set $H|_V=0.$
Further,
as $H$ is invariant under the circle action, we can decompose the surface $B$ into surfaces $B_-$ and $B_+$, where 
$$B_-=\{H|_V\leq0\}/S^1\hskip2mm\textrm{and}\hskip2mm B_+=\{H|_V\geq0\}/S^1$$
and  $B_- \cap B_+=\partial B_{\pm}=\{H|_V=0\}/S^1.$ 
Then, there exists an induced homeomorphism $\varphi: B_- \to B_+$ that is equal to identity on the set $\partial B_{\pm}$ and such that
$\pi_+ \circ \tilde \varphi=\varphi \circ \pi_-,$ where $\pi_{\pm}: V_{\pm} \to B_{\pm}$ are projections.
Since $H^2(B_-,\mathbb Z)=0$, as
$B_-$ has a non-empty boundary, it follows that  there exists a section $\sigma_{ - }:B_{ - }\rightarrow V_-.$ 
Then, the map $\sigma_+: B_+ \to V_+$ defined by $\sigma_+=\tilde \varphi \circ \sigma_- \circ \varphi^{-1}$ is also a section, since $\pi_+\circ\sigma_+$ is the identity on $B_+$. On the  set 
$B_- \cap B_+=\partial B_{\pm}$ it holds $\sigma_+=\sigma_- $ and, therefore,  the map $\sigma: B \to V$ defined by $\sigma|_{B_-}=\sigma_-$ and $\sigma|_{B_+}=\sigma_+$ is a global section on $B$.

Let us now show that there exists a global section $\tilde{\sigma}:W/S^1\rightarrow W$. First define $\tilde{\sigma}=\sigma$ on $B.$ Next, take an orbit $[x]\in W/S^1$ 
 of an interior point $x\in W.$ The gradient flow sends this orbit, in a positive time direction, to an orbit in $V_+,$ that is, $[\Phi^{\nabla H}_{t_x}(x)]\in B_+,$ where $t_x>0$ is the time when the trajectory in $x$ hits the boundary, i.e.
 when $\Phi^{\nabla H}_{t_x}(x)\in V_+.$ The map $[x]\mapsto [\Phi^{\nabla H}_{t_x}(x)]$ is well defined since $\nabla H$ and $X$ commute and, thus, the corresponding flows also commute. 
Note that the flow $\Phi^{\nabla H}_t$ of the point $\sigma([\Phi^{\nabla H}_{t_x}(x)])\in V_+$ in the negative time direction intersects the orbit of $x$ in some point in the interior of $W$. Indeed, 
$\sigma([\Phi^{\nabla H}_{t_x}(x)])\in V_+$ belongs to an orbit of $\Phi^{\nabla H}_{t_x}(x)$ and, as $X$ and $\nabla H$ commute, there is a point on its trajectory  that belongs to an orbit $O_x$ of $x.$ Therefore, we define 
$$\tilde\sigma([x])=\Phi_t^{\nabla H}(\sigma([\Phi^{\nabla H}_{t_x}(x)])) \cap O_x .$$
 Since $\tilde\sigma([x])\in W$ is a point that belongs to the orbit of $x$, then $\pi(\tilde\sigma([x]))=[x],$ where $\pi:W\to W/S^1$ is the projection. Therefore, 
  $\tilde{\sigma}:W/S^1\rightarrow W$ is a global section.

  The proof of Theorem  \ref{theorem1b} is completed.  
  
 \vskip2mm

\begin{remark}\label{remark:higher}  In higher dimensions, a symplectic manifold with a contact type boundary and a free Hamiltonian circle action may not be a trivial circle bundle. To construct an example, we will use
that for any Riemannian manifold $L$ with a free circle action, the disc cotangent bundle $D^*L$ can be equipped with  a free Hamiltonian circle action. Consider $\mathbb{RP}^n,$ for any $n\geq 2.$ Since the circle bundles over a manifold $B$ are classified by $H^2(B,\mathbb Z)$ and   $H^2(\mathbb{RP}^{n},\mathbb Z)=\mathbb Z/\mathbb Z_2,$
for every $n\geq 2,$
 it follows that there is a non-trivial circle bundle over $\mathbb{RP}^{n}.$
Denote by $L$ the total space of this bundle. The corresponding circle action on disc cotangent bundle $D^*L$ will be a Hamiltonian free circle action, but it will not be trivial.

Alternatively, consider for every $n\geq1$  a free circle action on a sphere $S^{2n-1}.$ (No free circle action exists on $S^{2n},$ because every vector field on $S^{2n}$ has to be zero somewhere.)
The lift of this action to a circle action on the disc cotangent bundle $D^*S^{2n-1}$ is a free Hamiltonian circle action. However, $D^*S^{2n-1}$ is not a trivial $S^1$-bundle. 
 Namely, a non-contractible loop in $D^*S^d$ would project to a non-contractible loop in $S^d$. Since $\pi_1(S^{d})=0,$ for every $d>1,$ it follows that $\pi_1(D^*S^d)=0,$ for every $d>1.$ 
Therefore, $D^*S^{2n-1}$ is not a trivial $S^1$ bundle.

 \end{remark}
 
%

%------------------------------------------------
\section{Existence of critical surfaces}\label{section:surface}
%----------------------------------------------------

\emph{Proof of Theorem \ref{theorem2}.}

Suppose $\Sigma$ is a closed surface. Take  the $\SS^1$-bundle over $\Sigma$ given by the first Chern class equal to $[\omega_{\Sigma}].$ According to Stoke's theorem, 
$[\omega_{\Sigma}]$ is a non zero 2-form and, therefore, this is a non trivial bundle. 
Moreover, the total space $V$ is a prequantization space, that is,  there exists  a connection form $\alpha$ that is also a contact form (\cite[Theorem 3]{BW}). The right translation on each fiber is an induced $S^1$ action on $V$ and we denote by $Z$ the infinithesimal generator of this action. Then $\iota_Zd\alpha=0$ and $\alpha(Z)=1.$
Consider the circle action on $V\times D^2_{(\rho, t)}$ defined by 
 $$e^{i\theta}\ast(p, \rho, e^{it})\mapsto (e^{-i\theta}*_Vp, \rho, e^{i(\theta+t)}), \hskip2mm  e^{i\theta}\in S^1,$$
  where $*_V$ denotes the $\SS^1$-action on  $V.$ Note that this action is generated by the vector field 
  $-Z+\frac{\partial}{\partial t}.$
Next, define a 2-form on $V\times D^2$ by
$$\omega=\rho d \rho\wedge\alpha+\rho d\rho\wedge dt+\frac{1+\rho^2}{2}d\alpha.$$
In every point of $V\times D^2$ the kernel of
$\omega\wedge\omega=-\rho(1+\rho^2 )d\rho\wedge d\alpha(\alpha+dt)$
is 1-dimensional space generated by the vector field $-Z+\frac{\partial}{\partial t}.$

Let $W$ be the quotient space of  $V\times D^2$  by the vector field $-Z+\frac{\partial}{\partial t}.$
Then,  $\omega$ is a symplectic form on  $W$. 
Further, a vector field defined by
$$Y=\frac{1+\rho^2}{2\rho}\frac{\partial}{\partial \rho}$$
is a Liouville vector field  for $\omega$ pointing transversally out of the boundary of $W.$ Hence, the boundary is of contact type.
Finally, consider the circle action on $D^2$ factor of $W$  generated by the vector field
$X=\frac{\partial}{\partial t}$. Since
$\iota_X\omega=d(\rho^2/2),$ the action is Hamiltonian. The set of critical points of the corresponding Hamiltonian function, given by the condition $\rho=0$ in $W$ is precisely the set of orbits $V/{S^1}$ and this is the surface $\Sigma.$

\vskip2mm

Suppose now $\Sigma$ has a non-empty boundary. As explained in Section \ref{section:free}, $\Sigma$ admits a structure of a Liouville domain with a symplectic form $\omega_{\Sigma}$ and a Liouville vector field
$Y_{\Sigma}.$ 
Then, $(\Sigma\times D^2,\omega_{\Sigma}+\frac{i}{2}dz\wedge d\overline{z})$ is a Liouville domain since  $i_*Y_{\Sigma}+j_*\nabla f$ is a Liouville vector field on $\Sigma\times D^2$ pointing transversally out of the boundary of $\Sigma\times D^2$, where $i: \Sigma\to\Sigma\times D^2$ and $j:D^2\to\Sigma\times D^2$ are inclusions and $f(z)=|z|^2$. 
As $\Sigma\times D^2$ is a manifold with corners, we proceed by smoothing the corners.
We  define a smooth function $g:\Sigma\rightarrow\mathbb{R}$ by $g(p,t)=1+\frac{t}{\epsilon}$ on the collar neighborhood $(-\epsilon,0]\times\partial\Sigma$ and by $g=0$ in the complement of the neighborhood  of this collar neighborhood. Note that in the collar neighborhood
it holds $Y_{\Sigma}=\frac{\partial}{\partial t}.$
Define $F:\Sigma\times D^2\rightarrow\mathbb{R}$ by $F=\tilde f+\tilde g$, where $\tilde f, \tilde g:\Sigma\times D^2\rightarrow\mathbb{R}$ are natural ectensions of the functions $f$ and $g$, and define
$$W=F^{-1}([0,1]).$$
 Then $W$ is a smooth manifold with the boundary
$\partial W=F^{-1}(\{1\}).$
The vector field $i_*Y_{\Sigma}+j_*\nabla f$ is transversal to $\partial W$ since
$T\partial W=\ker dF$ and $dF(i_*Y_{\Sigma}+j_*\nabla f)=dg(Y_{\Sigma})+df(\nabla f)>0$ along the boundary.  It follows that $(W,(\omega_{\Sigma}+\frac{i}{2}dz\wedge d\overline{z})|_W)$ is a symplectic manifold with a contact type boundary.
The $S^1$-action on $\Sigma\times D^2$ defined as the rotation on $D^2$ is a Hamiltonian action. The restriction to $W$ is also a Hamiltonian action and the fixed point set is precisely $\Sigma.$

The proof of Theorem \ref{theorem2} is completed.

\vskip2mm

\begin{remark}
Note that any closed Riemannian surface admits a symplectic structure (that can be made integral). Moreover, any Riemannian surface with boundary admits a structure of a Stein domain.

\end{remark}

\begin{example}\label{example: sphere} a) If $\Sigma=S^2$ and $\omega_{\Sigma}$ is the area form on the sphere, then the prequantization space is $S^3$ with the standard contact structure, i.e. the hyperplane field of complex tangencies. The resulting manifold constructed in the previous proposition is $S^3\times D^2/_{\sim}.$ Let $D^4\sharp\mathbb{CP}^2$ denote a symplectic blow up of $D^4,$ that is, an open ball around origin of a small radius $\epsilon$ is removed and the boundary $S^3_{\epsilon}$ of the ball is collapsed along the fibers of the Hopf fibration
$S^3_{\epsilon}\rightarrow \mathbb{CP}^1$ (for details of a blow up see Section \ref{section:blow}). Chose the radius $\epsilon$  in such a way that the area form of $\mathbb{CP}^1$ is precisely $\omega_{\Sigma}$.
Then, the map
$S^3\times D^2/_{\sim}\rightarrow D^4\sharp\mathbb{CP}^2$ given by $[z_1,z_2,z]\rightarrow [z_1z,z_2z]$ is a diffeomorphism.
The  circle action on $D^4\sharp\mathbb{CP}^2$  induced by the diagonal circle action on $D^4$ is Hamiltonian and $\mathbb{CP}^1=S^2$ is a fixed surface. See also Example \ref{example:disc} b).

b) If  $(\Sigma=D^2, \frac{i}{2}dz\wedge d\overline{z})$, following the construction explained above, we obtain that
$D^4$ with the $S^1$-action given by weights $(1,0)$ is the symplectic manifold with a contact type boundary with the critical surface  $D^2.$ See also Example \ref{example:disc} a).

c)  If $(\Sigma=S^1\times[-1,1], dh\wedge dt)$, then for
$(S^1\times D^3,dh\wedge dt+\frac{i}{2}dz\wedge d\overline{z})$ where $D^3=\{|z|^2+h^2\leq1\}$ with the $S^1$-action  given by $e^{i\theta}\ast(e^{it},z,h)\mapsto(e^{it},e^{i\theta}z,h)$
it follows that $\Sigma$ is the fixed surface.
See also Example \ref{example:second}.

\end{example}

%----------------------------------------------------
\section{Attaching a Weinstein Hamiltonian 2-handle}\label{section:handle}
%-------------------------------------------------------------
The main purpose of this section is to prove Theorem \ref{theorem3}.

Let $(W,\omega)$ be symplectic 4-manifold with a contact type boundary $V$ equipped with a Hamiltonian circle action such that the  set of Legendrian orbits is non-empty.
Then, according to Lemma \ref{lemma:legendrian_tori1}, both local extrema are attained at the boundary and in the interior there can be  only isolated critical points of index 2.
Let us attach a Weinstein 2-handle $\mathbb H$ to $(W,\omega)$ so that the resulting manifold  $(\tilde{W},\tilde \omega )$ is also $S^1$-Hamiltonian with a contact type boundary. 
To do so, we also need to impose a Hamiltonian circle action on the handle.
\vskip1mm

\subsection{ Hamiltonian circle action on a Weinstein 2-handle.}
For the model of a Weinstein 2-handle we take
a  subset $\mathbb H \subset  \mathbb R^4  $
defined as the intersection $ D  \cap S$ where  $D=   \{(x_1,x_2,y_1,y_2)\in \mathbb R \hskip2mm | \hskip2mm x_1^2+x_2^2    \leq1   \}$   and $S=    \{ (x_1,x_2,y_1,y_2)\in \mathbb R \hskip2mm | \hskip2mm y_1^2+y_2^2    \leq f(x_1^2+x_2^2) \} $, for some smooth function $f$, where $f(0)=1/3$ and whose graph is tangent to the   
$D$ close to the boundary, 
as shown in Figure    \ref{handle}. We equip $\mathbb H$ with the standard symplectic structure
$$\omega_0=\sum_{k=1}^2dx_k\wedge dy_k$$
 and the Liouville vector field
$$Y=\sum_{k=1}^2(-x_k\frac{\partial}{\partial x_k}+2y_k\frac{\partial}{\partial y_k}),$$
 depicted by blue arrows.  The corresponding Liouville 1-form is given by
$$\lambda=\sum_{k=1}^2(-x_kdy_k-2y_kdx_k)$$
 and since the boundary of $\mathbb H$ is transversal to $Y,$ it admits a contact structure $\ker\lambda.$

    \begin{figure}
\centering
\includegraphics[width=8cm]{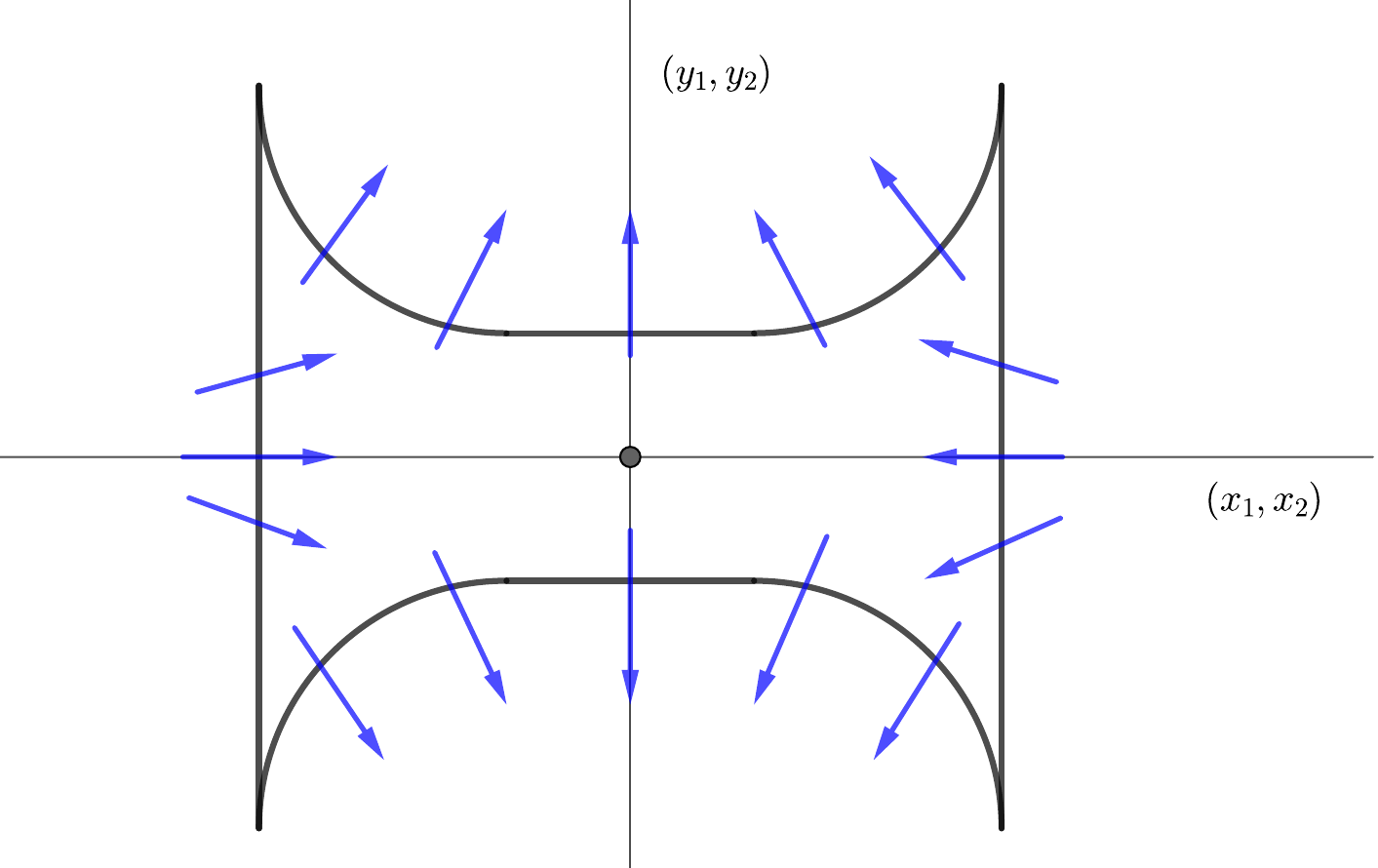}
\caption{A model for a Weinstein 2-handle}
\label{handle}
\end{figure}      

\vskip2mm

The core of the handle $\mathbb{H}$ is the stable manifold of the flow $\Phi^Y_t$ of the vector field $Y$ with respect to the critical point $\mathbb{O}=  (0,0,0,0).$ Since 
$\Phi^Y_t(x_1,x_2,y_1,y_2)=(e^{-t}x_1,e^{-t}x_2,e^{2t}y_1,e^{2t}y_2),$
  the core of the handle $\mathbb{H}$
  $$W^s(  \mathbb{O},\Phi^Y_t)=\{(x_1,x_2,y_1,y_2)\in \mathbb H | \lim_{t\rightarrow+\infty}\Phi^Y_t(x_1,x_2,y_1,y_2)=\mathbb{O}\}$$
is the following 2-disc
$ \{x_1^2+x_2^2\leq 1, y_1=y_2=0\}.$
Similarly, the co-core of the handle $\mathbb{H}$ is the 2-disc $\{ x_1=x_2=0, y_1^2+y_2^2 \leq f(0)=1/3\},$ being the unstable manifold of the point $\mathbb{O}.$

\vskip2mm

The boundary  of the handle $\mathbb H$ decomposes into two parts, $\mathbb H\cap\partial S$, where $Y$ points out the handle,  and $\mathbb H\cap\partial D$ where $Y$ points into  the handle.
The later is  called the attaching region.

\vskip2mm

Consider now the circle action on the handle defined by 
$$e^{i\theta}\ast_{(m,n)}(x_1+ix_2,y_1+iy_2)\mapsto(e^{im\theta}(x_1+ix_2), e^{in\theta}(y_1+iy_2)),$$
where $m,n\in\mathbb Z$ are relatively prime. This action  is well defined on the handle since it preserves the values of $x_1^2+x_2^2$ and $y_1^2+y_2^2.$

The vector field generating the given action is given by
 $$Z=-mx_2\frac{\partial}{\partial x_1}+mx_1\frac{\partial}{\partial x_2}-ny_2\frac{\partial}{\partial y_1}+ny_1\frac{\partial}{\partial y_2}.$$
 Since $L_Z\lambda=(m-n)(2y_1dx_2-2y_2dx_1+x_2dy_1-x_1dy_2)$ it follows that $\lambda$ is invariant under the given circle action if and only if $m=n.$ Being relatively prime we get
 $m=n=\pm1.$ 
 
 Suppose $m=n=1.$ Then
 $$\iota_Z\omega=-d(x_2y_1-x_1y_2),$$ 
 the action is Hamiltonian 
 and the moment map is
 $$H(x_1,x_2,y_1,y_2)=x_2y_1-x_1y_2.$$
 \vskip1mm

 The only interior critical point of a Hamiltonian function $H$ is $\mathbb{O}=(0,0,0,0)$ and from the matrix 
 $$Hess H(\mathbb{O})=\begin{bmatrix} 
0 & 0 & 0 & -1 \\
0 & 0 & 1 & 0 \\ 
0 &  1 & 0& 0 \\
-1 & 0&0&0
\end{bmatrix}$$
we conclude that the index of $\mathbb{O}$ is equal to 2.

\subsection{ An invariant attaching map} 
To specify the attaching of the Weinstein handle $\mathbb H$ to $(W,\omega )$ we  need to define  the attaching map, i.e.
an embedding
$\varphi:\mathbb H\cap\partial D\rightarrow V$ that preserves the corresponding contact structures but also the given Hamiltonian actions. The boundary of the core of the handle $\mathbb H$, i.e. the attaching sphere, is a Legendrian orbit $x_1^2+x_2^2=1,y_1=y_2=0$ of the introduced Hamiltonian action and we will map it onto one Legendrian orbit  in $V$. We then  extend this invariant attaching map to the map on the collar neighborhoods that will preserve corresponding Liouville structures and that will also be invariant under the given Hamiltonian circle actions in order to obtain a global Hamiltonian circle action on the glued manifold.
We divide the proof into several steps.

\vskip1mm

 \textbf{Step 1.} \emph{A standard neighborhood of a Legendrian orbit.}\\
 An orbit of a Hamiltonian circle action on a contact 3-manifold is called Legendrian if it is tangent to the contact structure. Thus, for every point $p\in V$ whose orbit is Legendrian
it holds $H(p)=\alpha(X)_p=0.$ We now describe the standard neighborhood of any Legendrian orbit.

\begin{lemma} Consider a Hamiltonian circle action on a compact contact 3-manifold $V.$ Then there exists a neighborhood of any Legendrian orbit $L$  that is $S^1$-contactomorphic to
$$(S^1\times D^2, \alpha_0=xdt+dy),$$
 where the circle action is given by
$$\theta*_0(t,x,y)\mapsto(\theta+t,x,y)$$ and
$L\cong S^1\times\{x=0, y=0\}$ in this neighborhood.
\end{lemma}
\begin{proof} A Legendrian orbit is a free orbit and according to the Slice theorem there is an $S^1\times D^2$ neighborhood of any free orbit where the action is linear on the first factor. 
Let $\alpha_1=H(x,y)dt+f(x,y)dx+g(x,y)dy$ be another $S^1$-invariant contact form on $S^1\times D^2$ giving the same orientation as $\alpha_0$ and 
so that $L$ is Legendrian orbit.
The contact condition on $L$ becomes $\alpha_1\wedge d\alpha_1=(g\frac{\partial H}{\partial x}-f\frac{\partial H}{\partial y})dy\wedge dx\wedge dt>0.$
However, Legendrian orbits do not appear isolated, they form 2-dimensional tori (Lemma \ref{lemma:legendrian_tori}). Thus, we can shrink the disc and assume that $H(0,y)=0,$ for all $(0,y)\in D^2.$ Thus 
$\frac{\partial H}{\partial y}(0,y)=0$ so $g(0,y)\neq0.$ So, we can divide $\alpha_1$ by $g$ and obtain another contact form, in the same contact structure,
$\alpha_1=H(x,y)dt+f(x,y)dx+dy,$ close to $L.$
If we denote $\alpha_s=(1-s)\alpha_0+s\alpha_1$, then
$$\alpha_s\wedge d\alpha_s=(s\frac{\partial H}{\partial x}+1-s)dy\wedge dx\wedge dt=s\alpha_1\wedge d\alpha_1+(1-s)\alpha_0\wedge d\alpha_0.$$ 
That is,  $\alpha_s\wedge d\alpha_s$ is a linear interpolation between two volume forms giving the same orientation, and, therefore,
$\alpha_s\wedge d\alpha_s$ is also a volume for every $s\in[0,1].$ According to Gray stability theorem the flow 
of the vector field $X_s,$ uniquely defined by
$$\alpha_s(X_s)=0, \hskip2mm\iota_{X_s}d\alpha_s=\dot{\alpha_s}(R_s)\alpha_s-\dot{\alpha_s},$$
where $R_s$ is the Reeb vector field for $\alpha_s,$ 
gives the  isotopy between $\alpha_0$ and $\alpha_1$. 

 It is left to show that this flow is invariant under the given circle action, that is $[Z_0,X_s]=0$, where  $Z_0=\frac{\partial}{\partial t}$ is the generator of the given circle action. 
 Since $\alpha_s(X_s)=0$, by using the Leibniz rule for the Lie derivative, we obtain
$$0=L_{Z_0}\alpha_s(X_s)=\iota_{X_s}L_{Z_0}\alpha_s+\alpha_s([Z_0,X_s])=\alpha_s([Z_0,X_s]).$$
 Similarly, since $\iota_{X_s}d\alpha_s=\dot{\alpha_s}(R_s)\alpha_s-\dot{\alpha_s}$, it follows that
$$L_{Z_0}\iota_{X_s}d\alpha_s=L_{Z_0}( \dot{\alpha_s}(R_s)\alpha_s)-L_{Z_0}\dot{\alpha_s},$$
that is,
$\iota_{X_s}L_{Z_0}d\alpha_s+\iota_{[Z_0,X_s]}d\alpha_s= \dot{\alpha_s}(R_s)L_{Z_0}\alpha_s+(L_{Z_0} \dot{\alpha_s}(R_s))\alpha_s-L_{Z_0}\dot{\alpha_s}$
and thus
$$\iota_{[Z_0,X_s]}d\alpha_s=(L_{Z_0} \dot{\alpha_s}(R_s))\alpha_s.$$
Next,
$L_{Z_0} \dot{\alpha_s}(R_s)=\iota_{R_s}L_{Z_0}\dot{\alpha_s}+\alpha_s([Z_0,R_s])=\alpha_s([Z_0,R_s]).$
Finally, from
$0=L_{Z_0}\alpha_s(R_s)=\iota_{R_s}L_{Z_0}\alpha_s+\alpha_s([Z_0,R_s])=\alpha_s([Z_0,R_s])$ it follows
$$\iota_{[Z_0,X_s]}d\alpha_s=0.$$
As $\alpha_s([Z_0,X_s])=0$, we conclude $[Z_0,X_s]=0.$ 

\end{proof}

 \textbf{Step 2.} \emph{Defining an invariant attaching map.}\\
We first note that there is a diffeomorphism $\phi: S^1\times D^2 \to \mathbb H\cap\partial D,$ that maps the coordinates $(x_1,x_2, \tilde y_1, \tilde y_2)$ to the coordinates
$(x_1,x_2,y_1,y_2)$ and, therefore,  $\phi^*\lambda=\sum_{k=1}^2(-x_kd \tilde y_k-2 \tilde y_kdx_k).$ In order to simplify the notation, in the rest we work with $S^1 \times D^2 $ with coordinates $(x_1,x_2,y_1,y_2)$ and a Liouville form $\lambda.$ 

Note that, since $x_1^2+x_2^2=1,$ the pair $(x_1,x_2)$ defines unique point on the unit circle, with argument $t,$ i.e. $x_1+ix_2=e^{it}.$

 \begin{lemma} The map
$\varphi:(S^1\times D^2,\ker\lambda)\rightarrow (S^1\times D^2\subset V, \ker\alpha_0)$
defined by
$$\varphi(x_1,x_2,y_1,y_2)=( t, -y_1\sin t-y_2\cos t,-y_1\cos t-y_2\sin t),$$
where $x_1+ix_2=e^{it}$, is an equivariant contactomorphism and desired attaching map.

 \end{lemma}
 
  \begin{proof} 
Since 
$$\varphi^{-1}(t,x,y)=( \cos t, \sin t,-y\cos t+x\sin t,-y\sin t-x\cos t)$$
one can easily check that $\varphi$ is a diffeomorphism.
Moreover, $\varphi^*\alpha_0=\lambda,$ i.e. $ \varphi$ is a contactomorphism.
Next, 
$$\varphi(e^{i\theta}\ast_{(1,1)}(x_1,x_2,y_1,y_2))=\theta\ast_0\varphi(x_1,x_2,y_1,y_2),$$
 i.e. $\varphi$ is $S^1$-invariant.
 
Finally, $\varphi$ maps the attaching sphere $S^1\times\{x_1=x_2=0\}$ of the handle, i.e. the Legendrian orbit in a point $(1,0,0,0),$ to the Legendrian orbit $L=S^1\times\{0\}\subset S^1\times D^2$ in $V.$
 
  \end{proof}

  \textbf{Step 3.} \emph{Extending the attaching map to an invariant symplectomorphism.}\\ 
 In order to obtain an invariant Liouville structure in the neighborhoad of the boundary of the manifold $(  \tilde W,   \tilde  \omega)$, we need to glue also the collar neighborhoods of these contact manifolds and take care of the action.
 As explained in Section   \ref{section:max}, there exists a boundary collar of $S^1\times D^2\subset V$
  diffeomorphic to $(-\epsilon,0]\times S^1\times D^2$ with the symplectic structure $d(e^{s} \alpha_0)$  and a Liouville vector field $\frac{\partial}{\partial s}$ pointing out of $V.$ On the other hand, there exists a boundary collar of the attaching region $S^1\times D^2\subset  \mathbb H$ diffeomorphic to    
  $[0,\epsilon)\times S^1\times D^2$ with the symplectic structure $d(e^{s} \lambda)$ and a Liouville vector field $\frac{\partial}{\partial s}$ pointing into the handle $  \mathbb H.$  
 We now extend $\varphi$ to the map
 $$\tilde{\varphi}:([0,\epsilon)\times S^1\times D^2,d(e^{s} \lambda))\rightarrow ((-\epsilon,0]\times S^1\times D^2\subset V, d(e^{s} \alpha_0))$$
defined by
 $$\tilde{\varphi}(s,p)=(s-\epsilon ,\varphi(p)).$$

Note that $\tilde{\varphi}$  is invariant with respect to the Hamiltonian circle actions, since the actions are constant in $s$-direction (see Equation   \ref{eq:action}). Moreover, the boundaries glued by $\varphi$ will be now glued smoothly.
Thus, described attaching of an $S^1$-Hamiltonian 2-handle to the $S^1$-Hamiltonian symplectic manifold $(W,\omega)$ with a contact type boundary $V$ is a new $S^1$-Hamiltonian symplectic
manifold $(\tilde{W},   \tilde   \omega)$.

  \textbf{Step 4.} \emph{Global Hamiltonian circle action.}\\ 
Let us find the Hamiltonian function on $\tilde{W}.$
If $X_W$ and $X_h$ denote the generators of the Hamiltonian actions on $W$ and on the handle, respectively, and if $H_W$ and $H_h$ are the corresponding Hamiltonian functions
 then, on the glued part of the boundary, we have
$$H_h=\lambda(X_h)=\tilde{\varphi}^*\lambda(X_h)=\alpha_0(d\tilde{\varphi}(X_h))=\alpha_0(X_W)=H_W,$$
since $\tilde{\varphi}(X_h)=X_W.$ Thus the Hamiltonian function on $\tilde{W}$ being defined by
$\tilde{H}|_W=H_W$ and $\tilde{H}|_h=H_h$ is well defined.

The proof of Theorem  \ref{theorem3} is completed.

\begin{example} a) Consider $(S^1\times D^3, dh\wedge dt+\frac{i}{2}dz\wedge d\overline{z})$ with the Liouville vector field $Y=h\frac{\partial}{\partial h}+\frac{1}{2}(z\frac{\partial}{\partial z}+\overline{z}\frac{\partial}{\overline{z}})$ and with the free Hamiltonian circle action. 
Let us attach a Hamiltonian 2-handle in the manner described above. The set of Legendrian orbits in the boundary of $S^1\times D^3$ is the torus $S^1\times \{h=0\}\subset S^1\times S^2.$ Take any such orbit.
Note that 
there is a Weinstein handle decomposition of $S^1\times D^3$ into one 0-handle and one 1-handle. In this decomposition,  the chosen Legendrian orbit  intersects the belt sphere of the 1-handle $0\times S^2\subset D^1\times D^3$ in a single point. Since we are mapping the attaching sphere to this Legendrian orbit, we conclude that the attaching sphere of the 2-handle intersects the belt sphere of the 1-handle in a single point. Thus, they can be canceled (i.e. they fill together $D^4$) and the resulting manifold is just a 0-handle, that is
$D^4$ with the standard symplectic form $\omega_{st}=\sum_{k=1}^2dx_k\wedge dy_k.$ Since the resulting manifold contains one isolated critical point of index 2,
we conclude that by gluing Hamiltonian 2-handle to $S^1\times D^3$ we obtain $(D^4,\omega_{st})$ with the Hamiltonian circle action given by the weights $(1,-1).$
\vskip1mm
b) Let us attach a Weinstein 2-handle to $(D^4,\omega_{st})$ with the Hamiltonian circle action given by the weights $(1,-1).$ 
 The set of Legendrian orbits in the boundary of $D^4$  is the torus $\{(z_1,z_2)\in D^4 \hskip1mm | \hskip1mm |z_1|=|z_2|=1/\sqrt 2\}.$ Take  the Legendrian orbit 
$$L=\Big\{(\frac{1}{\sqrt 2}e^{i\theta},\frac{1}{\sqrt 2} e^{-i\theta})\in S^3\hskip2mm | \hskip2mm\theta\in S^1\Big\}.$$
This  orbit bounds a Lagrangian disc
$D=\Big\{(\frac{\rho}{\sqrt 2}e^{i\theta},\frac{\rho}{\sqrt 2} e^{-i\theta})\in D^4\hskip2mm | \hskip2mm\theta\in S^1, \rho \in [0,1]\Big\}$ in $(D^4,\omega_{st})$ and  the standard neighborhood of  $D$,  the disc cotangent bundle  $D^*D$ with the canonical structure  is Weinstein homotopic to $(D^4,\omega_{st})$.  As a Weinstein 2-handle  is Weinstein homotopic to  $D^*D^2$ where the zero section is the core of the handle, the attaching map is a gluing of zero sections  and, thus, we obtain the cotangent bundle of $S^2.$
(Alternatively, from the Gompf-Kirby theory, if 
the framing coefficient  of the attaching Legendrian circle in the boundary of $D^4$ is equal to $-1$, then  the manifold obtained is the disc cotangent bundle of the sphere, $D^*S^2$ and, as explained in \cite[Example 3.5.7]{Geiges}), the framing coefficient  of $L=\Big\{(x_1,y_1,x_2,y_2)\in S^3\hskip2mm | \hskip2mmx_1^2+x_2^2=1, y_1^2+y_2^2=0\Big\}$  is precisely -1.) For more details on Weinstein handlebody structure of cotangent bundles over surfaces we refer to \cite[Section 7]{ACSG}. Therefore, by attaching a Weinstein 2-handle to $D^4$ along $L$ we obtain a Weinstein domain that is Weinstein homotopic to $D^*S^2$ with two isolated critical points both with index 2.
 Note that this manifold can also be obtained by taking the cosphere bundle of $S^2$ with the $T^2$ action as shown in Example \ref{example:cotangent}.

If we keep attaching Weinstein 2-handles to $D^*S^2$ we obtain a linear plumbing of disc cotangent bundles of spheres. The number of isolated critical points of the corresponding Hamiltonian action is equal to the number of 2-handles plus 2.

\vskip1mm
c) Consider again $(S^1\times D^3, \omega=dh\wedge dt+\frac{i}{2}dz\wedge d\overline{z})$ with a Hamiltonian action $(k=2,m=1)$ as in Example \ref{example:second}. The corresponding Hamiltonian function is given by
$H(h,e^{it},z)=2h+\frac{1}{2}|z|^2$ and, thus, 
the set of Legendrian orbits is the torus $S^1\times\{h=2-\sqrt 5\}.$
Note that the action is not free, namely, there exists one $\mathbb{Z}_2$ gradient annulus, $S^1\times \{z=0\}$. Let us attach a Hamiltonian Weinstein 2-handle. Take the Legendrian orbit
$$L=\Big\{(e^{2i\theta}, h=c, z=de^{i\theta})\in S^1\times S^2 \hskip2mm |\hskip2mm \theta\in S^1\Big\},$$
where $c=2-\sqrt5$ and $d^2=4\sqrt5-8.$ 
This orbit bounds $M=\Big\{(e^{2i\theta}, h=c, z=d\rho e^{i\theta})\in S^1\times D^3 \hskip2mm |\hskip2mm \theta\in S^1, \rho \in[0,1]\Big\}$ what is topologically a Möbius band. Moreover,  $M$ is a Lagrangian submanifold in  $(S^1 \times D^3,\omega)$ and  $D^*M$ with the canonical Weinstein structure  is Weinstein homotopic to $(S^1 \times D^3,\omega)$. Then, by attaching a Weinstein 2-handle to it, we glue the core (the zero section of $D^*D^2$)  along the boundary of the Möbius band, and again the attaching map represents the gluing of the zero sections. Therefore,  we obtain a Weinstein domain that is Weinstein homotopic to $D^*\mathbb{RP}^2$ with one isolated critical point of index 2 and one $\mathbb{Z}_2$-gradient annulus.
 Note that this manifold can also be obtained by taking the cosphere bundle of $\mathbb{RP}^2$ with the $T^2$ action as shown in Example \ref{example:cotangent}.

\vskip1mm

d)
Consider the disc cotangent bundle of the torus $(D^*T^2=T^2\times D^2, dx\wedge dt_1+dy\wedge dt_2)$ with the Liouville vector field $Y=x\frac{\partial}{\partial x}+y\frac{\partial}{\partial y}$
and with the free Hamiltonian circle action 
$$e^{i\theta}\ast(e^{it_1},e^{it_2}, x,y)\mapsto(e^{i(t_1+m_1\theta)},e^{i(t_2+m_2\theta)},x,y),$$
where $m_1,m_2\in\mathbb Z$ are relatively prime.
The corresponding Hamiltonian function is 
 $H(t_1,t_2,x,y)=m_1x+m_2y.$ The set of Legendrian orbirs is the union of two tori 
 $T^2\times\{x=\frac{m_2}{\sqrt{m_1^2+m_2^2}}, y=\frac{-m_1}{\sqrt{m_1^2+m_2^2}}\}$
 and  
  $T^2\times\{x=\frac{-m_2}{\sqrt{m_1^2+m_2^2}}, y=\frac{m_1}{\sqrt{m_1^2+m_2^2}}\}$
and every Legendrian orbit is uniquely given by the slope $(m_1,m_2)$ in the torus $T^2$ and by  the point $(x,y)\in\partial D^2.$ 
 Attaching a Weinstein 2-handle to $D^*T^2$ where the attaching curve is given by $(m_1,m_2)\in T^2$ we obtain a Weinstein domain that is Weinstein homotopic to
 the self-plumbing of the disc cotangent bundle of $S^2$ as shown in \cite[10.1.]{ACSG}. (Note that it does not depend on the choice of constants $(m_1,m_2)$.)
The resulting Hamiltonian manifold has only one isolated critical point of index 2.

\end{example}

\begin{remark}(Attaching a Hamiltonian 1-handle)

We now briefly explain that by attaching Hamiltonian 1-handle we obtain the manifold that was already obtained by applying the method in Theorem \ref{theorem2}.  

Consider a symplectic 1-handle
$(D^1_{y_2}\times D^3_{(x_1,x_2,y_1)},\omega_0=\sum_{k=1}^2dx_k\wedge dy_k)$ with the Liouville vector field
$$Y=\frac{1}{2}x_1\frac{\partial}{\partial x_1}+\frac{1}{2}y_1\frac{\partial}{\partial y_1}+2x_2\frac{\partial}{\partial x_2}-y_2\frac{\partial}{\partial y_2}.$$

Then the Liouville 1-form is
$$\lambda=\frac{1}{2}x_1dy_1-\frac{1}{2}y_1dx_1+2x_2dy_2+y_2dx_2.$$
Note that $Y=\nabla f$ where $f(x_1,y_1,x_2,y_2)=\frac{1}{4}x_1^2+\frac{1}{4}y_1^2+x_2^2-\frac{1}{2}y_2^2.$
Gradient trajectories defined by
$$\frac{d\psi_t}{dt}|_{t=t_0}=Y_{\psi(t_0)}$$
are
$$\psi_t(x_1,y_1,x_2,y_2)=(e^{\frac{t}{2}}x_1,e^{\frac{t}{2}}y_1,e^{2t}x_2,e^{-t}y_2).$$
The core of the handle is the stable manifold of the critical point $(0,0,0,0)$ and this is
$$W^s((0,0,0,0))=\{x_1=x_2=y_1=0\}=D^1\times\{0\}.$$
The cocore of tha handle is the unstable manifold and this is
$$W^u((0,0,0,0))=\{y_2=0\}=\{0\}\times D^3.$$
\vskip5mm
Consider the circle action on the handle defined by
$$e^{i\theta}\ast_{m}(y_2,x_1+iy_1,y_1)\mapsto(y_2,e^{im\theta}(x_1+iy_1), x_2),$$
where $m\in\{1,-1\}$, in order to have an effective action.

The vector field generating the action is given by
 $$Z=-my_1\frac{\partial}{\partial x_1}+mx_1\frac{\partial}{\partial y_1}.$$
 Then  $L_Z\lambda=0$ and $\lambda$ is invariant under the given circle action.
 
 Next, we compute the corresponding Hamiltonian function
 $$H(x_1,x_2,y_1,y_2)=\frac{m}{2}(x_1^2+y_1^2).$$

There is a fixed surface: 
 $$D^1_{y_2}\times D^1_{x_2}\times\{x_1=y_1=0\}\subset D^1_{y_2}\times D^3_{(x_1,x_2,y_1)}.$$
 This fixed surface is a product of the core of the handle and an interval. Hence, this surface has a non-empty boundary.
 
In order to glue such a 1-handle, we need that the Hamiltonian $S^1$-action on a manifold $(W,\omega)$ with a contact type boundary also admits a fixed surface with a non-empty boundary.
Let us denote such a surface by $\Sigma_g^n,$ where $g$ is a genus and $n$ is the number of boundary components. If it corresponds to the maximum (minimum), then we are able to glue a Hamiltonian 1-handle to 
  $(W,\omega)$ with $m=1$ ($m=-1$). We obtain new Hamiltonian $S^1$ manifold $(W',\omega')$ with a fixed surface $\Sigma_{g+1}^{n-1}.$ That is, the genus increases, however, there will be one boundary component less.

\end{remark}

%-------------------------------------------------------------
\section{Blow up of an interior fixed point}\label{section:blow}
%-----------------------------------------------------------

This section is mostly adopted from Karshon \cite[Section 6]{Karshon} since the procedure of a blow up of sufficiently small size  does not change the contact boundary.

\vskip2mm

\emph{Proof of Theorem \ref{theorem4}.}

Let $(W,\omega)$ be any symplectic 4-manifold with a contact type boundary and a Hamiltonian circle action and  
let us perform a symplectic blow up of $W$ at an interior critical point $x.$

 $\bullet$ Suppose $x\in W$ is an isolated critical point. (Recall also that critical points at the boundary are never isolated, they appear in circles.)

According to equivariant Darboux theorem, there is a small neighborhood around $x$ in the interior of $W$
that is equivariantly symplectomorphic to the disc of a small radius
$(D^4, \omega_{st}=\frac{i}{2}\sum_{k=1}^2dz_k\wedge d\overline{z}_k)$ with a Liouville vecor field  $Y=\frac{1}{2}\sum_{k=1}^2(z_k\frac{\partial}{\partial z_k}+\overline{z}_k\frac{\partial}{\overline{z}_k}))$ and with the Hamiltonian circle action given by weights $(m,n)$ (see Example \ref{example:disc}), 
where $m,n\in\mathbb{Z}$ are relatively prime and $mn\neq0$.  (If $mn=0$ then, there $x$ belongs to a critical surface, so it is not an isolated critical point.)
The corresponding Hamiltonian function is given by 
$H(z_1,z_2)=\frac{1}{2}(m|z_1|^2+n|z_2|^2)$
 and the point $(0,0)$ is an isolated critical point corresponding to $x$.

We perform a symplectic blow up at $(0,0),$ 
by removing a ball around $(0,0)$ of a small radius $\epsilon$ and collapsing the boundary $S^3_{\epsilon}$ along the fibers of the Hopf fibration, i.e. identifing the points in $S^3_{\epsilon}$ by $(z_1,z_2)\sim(e^{i\theta}z_1,e^{i\theta}z_2).$ The orbit space obtained is an exceptional sphere and it is symplectomorphic to $\mathbb{CP}^1$ with a Fubini-Study symplectic form. A circle action on $D^4$ defined by the weights $(m,n)$ induces a circle action on the blow up $D^4\sharp\mathbb{CP}^2$.

 If $mn=1$
 every point of the exceptional sphere is fixed and in this way we obtain a critical surface $S^2$ (see also Example \ref{example: sphere}).

If $mn\neq1$ then there are two fixed points in the exceptional sphere $\mathbb{CP}^1,$ namely $[\epsilon:0]$ and $[0:\epsilon].$ 
 We distinguish the following cases.
 If $mn>0,$ i.e $(0,0)$ is a local extremum,  one fixed point will be local extremum, the other will have index 2. Precisely, if $n>m>0$ then $[\epsilon:0]$ is a local minimum; if $m>n>0$ then $[0:\epsilon]$ is a local minimum and similarly if $m,n<0.$
 If $mn<0$ i.e. $(0,0)$ has index 2, then both fixed points will have index 2. 
Note also that all the points of the exceptional sphere $\mathbb{CP}^1$  are fixed by $\mathbb{Z}_{k=|m-n|}$.
Namely,
take any $e^{2\pi i/k},$ where $0\leq i\leq k-1.$ Assume $m>n,$ then $k=m-n$ and it follows that
$$e^{\frac{2i\pi}{k}}*(z_1,z_2)\mapsto(e^{\frac{2mi\pi}{k}}z_1,e^{\frac{2ni\pi}{k}}z_2)\sim (e^{\frac{2ki\pi}{k}}z_1,z_2)=(z_1,z_2).$$
Therefore, the exceptional sphere of a blow up is a $\mathbb{Z}_k$ gradient sphere.

\vskip2mm

$\bullet$ Suppose  $x\in\Sigma$, where $\Sigma\in W$ is a critical surface.

According to equivariant Darboux theorem, there is a small neighborhood around $x$ in the interior of $W$
that is equivariantly symplectomorphic to the disc of a small radius
$(D^4, \omega_{st}=\frac{i}{2}\sum_{k=1}^2dz_k\wedge d\overline{z}_k)$ with a Liouville vecor field  $Y=\frac{1}{2}\sum_{k=1}^2(z_k\frac{\partial}{\partial z_k}+\overline{z}_k\frac{\partial}{\overline{z}_k}))$ and with the Hamiltonian circle action given by weights $(m,0)$ or $(0,m)$,
 where $m=\pm1.$ If the weights are $(m,0)$, then the neighborhood of a critical point $x$ in a critical surface $\Sigma$ is a disc $\{0\}\times D^2.$

Suppose that  $x$ corresponds to $(0,0)$ and perform a symplectic blow up at $(0,0)$ of a small size $\epsilon>0.$
After removing the small disc $(0,z_2),$ where $|z_2|<\epsilon,$ inside the disc $\{0\}\times D^2,$ and collapsing the boundary of the small disc via $(0,e^{it}\epsilon)\sim(0,e^{i(\theta+t)}\epsilon)$  we conclude that obtained neighborhood will be also homeomorphic to a 2-disc and it has one intersection point $[0:\epsilon]$ with an exceptional sphere. As above  the
circle action on $D^4$ induces a circle action on the blow up $D^4\sharp\mathbb{CP}^2.$ Obtained $D^2$ disc will be a pointwise fixed and a point
 $[\epsilon:0]$ will be an isolated critical point.
The exceptional sphere is a gradient sphere with a trivial stabiliser. Therefore, in  $D^4\sharp\mathbb{CP}^2$ there is a critical surface diffeomorphic to $\Sigma$ and one additional critical point.

One may proceed blowing up a critical surface $k\geq1$ times and obtain $k$ gradient spheres whose north (south) poles belong to the surface, if the surface is a local maximum (minimum).

 \subsection*{Acknowledgments}
 The author is very grateful to Klaus Niederkrüger and Laura Starkston for numerous useful discussions.  
This research is partially supported by the Ministry of Education, Science and Technological Development, Republic of Serbia, through the project 451-03-66/2024-03/200104.


\begin{thebibliography}{9}

 \bibitem[ACSG+]{ACSG} 
  B. Acu, O. Capovilla-Searle, A. Gadbled , A. Marinković, E. Murphy, L. Starkston and A. Wu.
 \emph{ Weinstein handlebodies for complements of smoothed toric divisors.} 
   Memoirs of the American Mathematical Society- accepted.

 \bibitem[AH12]{AH} 
K. Ahara and A. Hattori. 
 \textit{ 4-dimensional symplectic $S^1$-manifolds admitting moment map.}
 J. Fac. Sci. Univ. Tokyo Sect. IA Math. vol 38 (1991), no. 2, 251–298.
 

 \bibitem[At82]{Atiyah} 
M. F. Atiyah.
  \textit{Convexity and commuting Hamiltonians.} 
Bull. London Math. Soc. vol 14 (1982), no. 1, 1–15.

  


  
 \bibitem[Au04]{Audin} 
  M. Audin.
  \textit{Torus actions on symplectic manifolds. 2nd revised ed..} 
Progress in Mathematics (Boston, Mass.) 93. Basel: Birkhäuser, 2004.


 \bibitem[BW58]{BW} 
W. M. Boothby and H. C. Wang.
  \textit{On contact manifolds.} 
Annals of Mathematics.
vol 68 (1958),  no. 3. 721-734.

  



 \bibitem[De88]{Delzant} T. Delzant.
   \textit{ Hamiltoniens périodiques et images convexes de l’application moment.} 
   Bull. Soc. Math. France. vol 116 (1988), no. 3, 315–339.


  
 \bibitem[G08]{Geiges} 
  H. Geiges.
  \textit{An introduction to Contact Topology.} 
Cambridge Studies in Advanced Mathematics, vol 109, Cambridge University Press, Cambridge, 2008.

 \bibitem[GS82]{GS} 
 V. Guillemin and S. Sternberg.
   \textit{Convexity properties of the moment mapping.}
    Invent. Math. vol 67 (1982), no. 3, 491–513.

 \bibitem[HK19]{HK} 
 T. S. Holm and L. Kessler.
   \textit{ Circle actions on symplectic four-manifolds.}
Commun. Anal. Geom.
vol 27, (2019), no. 2, 421–464.

       \bibitem[JR97]{JR} 
J.D.S. Jones and J.H. Rawnsley.
   \textit{           Hamiltonian circle actions on symplectic manifolds and the signature.}      
      J. Geom. Phys.
vol 23 (1997),  301-307.      
      
      
         \bibitem[KL97]{KL}    
   Y. Karshon and E. Lerman.
     \textit{    Hamiltonian torus actions on symplectic orbifolds and toric varieties}   
      Trans. Am. Math. Soc.      
vol 349 (1997), no. 10,  4201–4230.   



    \bibitem[Ka99]{Karshon}    
      Y. Karshon.
     \textit{    Periodic Hamiltonian flows on four-dimensional manifolds.}
      Mem. Amer. Math. Soc. vol 141 (1999), no. 672, viii+71.
      
      
         
         \bibitem[KZ18]{KZ}    
    Y. Karshon and F. Ziltener.
     \textit{    Hamiltonian group actions on exact symplectic manifolds with proper momentum maps are standard.}  
       Trans. Am. Math. Soc.       
vol 370 (2018), no. 2, 1409-1428.

      
      
        \bibitem[Li07]{Li}     
      H. Li.
   \textit{ The fundamental group of symplectic manifolds with Hamiltonian Lie group actions.}
J. Symplectic Geom. vol 4 (2007), no. 3, 345–372. 
     
      
       \bibitem[Lu79]{Lutz}
         R. Lutz. 
          \textit{        Sur la géométrie des structures de contact invariantes.}
           Ann. Inst. Fourier (Grenoble) vol 29 (1979), no. 1, 283–306.


        
 \bibitem[MN]{MN}
  A. Marinković and K. Niederkrüger. 
   \textit{   Symplectic circle actions on manifolds with contact type boundary.} arXiv:2203.08288.
   
    \bibitem[McD88]{McDuff}
     D. McDuff.
       \textit{   The moment map for circle actions on symplectic manifolds.}
        J. Geom. Phys. vol 5 (1988), no. 2, 149–160.   
   
       \bibitem[McDS17]{MS}
       D. McDuff and D. Salamon.
         \textit{    Introduction to symplectic topology.} 
         Oxford Mathematical Monographs. New York, NY: Oxford University Press.  (2017).
         
           
       \bibitem[McDT06]{MT} 
              D. McDuff and S. Tolman.
         \textit{   Topological properties of Hamiltonian circle actions.}
         IMRP Int. Math.
Res. Pap. (2006) 72826, 1–77.
                  
         
                \bibitem[Nie05]{Nie}     
                    K Niederkrüger.
    \textit{  Compact Lie group actions on contact manifolds.} (2005) PhD thesis.        
    
     
    
              \bibitem[T17]{Tolman}     
              S. Tolman.
              \textit{   Non-Hamiltonian actions with isolated fixed points.}
               Invent. Math.    vol 210 (2017),  877–910.
    
    

\end{thebibliography}
\end{document}